\documentclass[10pt]{article}

\usepackage{latexsym}
\usepackage{amssymb}
\usepackage{amsthm}
\usepackage{amscd}
\usepackage{amsmath}
\usepackage{mathrsfs}
\usepackage{graphicx}
\usepackage{hyperref}
\usepackage{shuffle}
\usepackage{mathtools}
\usepackage[bbgreekl]{mathbbol}
\usepackage{mathdots}
\usepackage{ytableau}
\usepackage{enumerate}
\usepackage[all]{xy}
\usepackage{tikz-cd}
\usepackage[colorinlistoftodos]{todonotes}
\usetikzlibrary{chains,scopes}
\usetikzlibrary{shapes.geometric,positioning}

\DeclareSymbolFontAlphabet{\mathbb}{AMSb}
\DeclareSymbolFontAlphabet{\mathbbl}{bbold}

\newcommand{\SetSSYT}{\ensuremath\mathrm{SetSSYT}}

\newcommand{\SetSSMT}{\ensuremath\mathrm{SetSSMT}}

\numberwithin{equation}{section}
\allowdisplaybreaks[1]

\theoremstyle{definition}
\newtheorem* {theorem*}{Theorem}
\newtheorem* {conjecture*}{Conjecture}
\newtheorem{theorem}{Theorem}[section]
\newtheorem{thmdef}[theorem]{Theorem-Definition}

\newtheorem{problem}[theorem]{Problem}
\theoremstyle{definition}

\newtheorem* {remark*}{Remark}
\newtheorem* {example*}{Example}

\newtheorem{lemma}[theorem]{Lemma}
\theoremstyle{definition}
\newtheorem{definition}[theorem]{Definition}
\theoremstyle{definition}
\newtheorem* {notation}{Notation}

\newtheorem{proposition}[theorem]{Proposition}
\newtheorem{corollary}[theorem]{Corollary}

\newtheorem{remark}[theorem]{Remark}
\theoremstyle{definition}
\newtheorem {example}[theorem]{Example}
\theoremstyle{definition}

\theoremstyle{definition}

\theoremstyle{definition}

\def\({\left(}
\def\){\right)}
\newcommand{\sP}{\mathscr{P}}
\newcommand{\CC}{\mathbb{C}}
\newcommand{\QQ}{\mathbb{Q}}

\def\NN{\mathbb{N}}

\def\Hom{\mathrm{Hom}}

\def\CC{\mathbb{C}}

\def\ZZ{\mathbb{Z}}

\def\GL{\textsf{GL}}

\def\spanning{\textnormal{-span}}

\def\fk{\mathfrak}

\def\barr{\begin{array}}
\def\earr{\end{array}}
\def\ba{\begin{aligned}}
\def\ea{\end{aligned}}
\def\be{\begin{equation}}
\def\ee{\end{equation}}

\def\qquand{\qquad\text{and}\qquad}
\def\quand{\quad\text{and}\quad}

\def\I{I}

\def\cH{\mathcal H}

\def\ds{\displaystyle}

\def\id{\mathrm{id}}
\def\PP{\mathbb{P}}
\def\MM{\mathbb{M}}

\def\fkS{\fk S}
\def\fkG{\fk G}

\def\ben{\begin{enumerate}}
\def\een{\end{enumerate}}

\def\cE{\mathcal E}

\def\fpf{{\textsf {FPF}}}

\newcommand{\wfpf}{\Theta}

\def\SS{\mathrm{SD}}

\newcommand{\Fl}{\textsf{Fl}}

\renewcommand{\O}{\mathsf{O}}
\newcommand{\Sp}{\mathsf{Sp}}
\newcommand{\K}{\mathsf{K}}

\newcommand{\Ess}{\operatorname{Ess}}
\newcommand{\rank}{\operatorname{rank}}
\renewcommand{\dim}{\operatorname{dim}}
\newcommand{\codim}{\operatorname{codim}}
\newcommand{\pf}{\operatorname{pf}}

\newcommand{\cA}{\mathcal{A}}

\def\cAfpf{\cA_\fpf}

\def\CK{C\hspace{-0.2mm}K}
\def\GP{G\hspace{-0.2mm}P}
\def\GQ{G\hspace{-0.2mm}Q}

\def\HAfpf{\mathcal{B}_\fpf}

\def\YY{\mathrm{D}}

\def\cHfpf{\cH^{\textsf{Sp}}}

\def\iG{\fkG^{\textsf{O}}}

\def\iS{\hat \fkS}

\def\arcstart{\ \xy<0cm,.2cm>\xymatrix@!@R=.1cm@C=.2cm }

\def\ellfpf{\ell_\fpf}

\def\cF{\mathcal{F}}
\def\cG{\mathcal{G}}

\def\Gr{\mathrm{Gr}}

\newcommand{\Sym}{\textsf{Sym}}

\def\cW{\mathcal{W}}
\def\cV{\mathcal V}
\def\cU{\mathcal U}

\def\cHfpf{\cH_{\Sp}}

\def\Gfpf{\mathfrak{G}^\Sp}

\newcommand{\KG}{\mathfrak{G}^\K}

\def\GSp{\GP^{\Sp}}
\def\GO{\GQ^{\O}}

\newcommand{\cN}{\mathcal{N}_\infty}

\def\Nil{\mathcal{U}_\infty}
\def\*{\mathbin{\hat\circ}}

\def\SLambda{\sP_{\textsf{strict}}}
\def\ILambda{I\hspace{-0.2mm}\Lambda}

\def\betaMin{\alpha_\fpf}

\def\varpi{\partial^{(\beta)}}
\def\vartheta{\pi^{(\beta)}}

\def\iX{X^{\O}}
\def\fX{X^{\Sp}}
\def\KX{X^{\K}}
\def\iMX{M\hspace{-0.2mm}X^{\O}}
\def\fMX{M\hspace{-0.2mm}X^{\Sp}}
\def\KMX{M\hspace{-0.2mm}X^{\K}}

\def\LG{\mathsf{LG}}
\def\OC{\fk C}

 \def\SWNEleq{\preceq}
\def\SWNEneq{\prec}
\def\graph{\Gamma}

\def\shO{\lambda^{\O}}

\def\DO{D^{\O}}
\def\DSp{D^{\Sp}}
\def\DK{D^\K}
\def\Mat{\mathsf{Mat}}
\def\MK{\mathsf{Mat}^\K}
\def\MO{\mathsf{Mat}^{\O}}
\def\MSp{\mathsf{Mat}^{\Sp}}

\newcommand{\Hs}{\mathcal{H}}

\def\ISp{I^{\textsf{FPF}}}

\def\pt{\mathrm{pt}}

\def\top{\mathrm{T}}
\def\wfpf{{1_{\fpf}}}
\begin{document}
\title{$K$-theory formulas for orthogonal and symplectic orbit closures}
\author{
Eric Marberg \\ HKUST \\ {\tt eric.marberg@gmail.com}
\and
Brendan Pawlowski \\ University of Southern California \\ {\tt br.pawlowski@gmail.com}
}

\date{}

\maketitle

\begin{abstract}
 The complex orthogonal and symplectic groups both act on the complete flag variety with finitely many orbits. We study two families of polynomials 
 introduced by Wyser and Yong representing the $K$-theory classes of the closures of these orbits. Our polynomials are analogous to the Grothendieck polynomials representing $K$-classes of Schubert varieties, and we show that like Grothendieck polynomials, they are uniquely characterized among all polynomials representing the relevant classes by a certain stability property. We show that the same polynomials represent the equivariant $K$-classes of symmetric and skew-symmetric analogues of Knutson and Miller's matrix Schubert varieties. We derive explicit expressions for these polynomials in special cases, including a Pfaffian formula relying on a more general degeneracy locus formula of Anderson.  Finally, we show that taking an appropriate limit of our representatives recovers  the $K$-theoretic Schur $Q$-functions of Ikeda and Naruse.
\end{abstract}

%\setcounter{tocdepth}{2}
%\tableofcontents

\section{Introduction}

Our results in this paper concern two families of polynomials 
representing $K$-theory classes of orbit closures in the complete flag variety,
which  we call \emph{orthogonal} and \emph{symplectic Grothendieck polynomials}.
For motivation, we start by reviewing
the classical story of \emph{Grothendieck polynomials}, which 
represent the $K$-theory classes of type A Schubert varieties.

Let $n$ be a positive integer
and write $\GL_n  =\GL_n(\CC)$ for the general linear group of 
invertible $n\times n$ complex matrices. Define $B\subseteq \GL_n$
to be the Borel subgroup of invertible lower triangular matrices.

Suppose  $X$ is a smooth complex algebraic variety.
Let $K(X)$ denote
the Grothendieck group of coherent sheaves on $X$ equipped with a ring structure induced by the (derived) tensor product. This the usual \emph{$K$-theory ring} of $X$.

We write $\CK(X)$ for the 
\emph{connective $K$-theory ring} of $X$ introduced by Cai \cite{Cai}.
This is a certain graded algebra over the coefficient ring $\ZZ[\beta]$,
which can be interpreted as the connective $K$-theory ring of
a point. For any closed equidimensional subscheme $Y \subseteq X$,
there is an associated $K$-theory class $[Y]_K \in K(X)$, namely the class of the structure sheaf of $Y$, and an associated
connective $K$-theory class $[Y]_{\CK} \in \CK(X)$.

We define 
 the \emph{complete flag variety} $\Fl_n := B\backslash\GL_n$ 
 to be the set of right cosets of $B$ in $\GL_n$.
 The ordinary $K$-theory ring of $\Fl_n$ can be realized as 
\be
K(\Fl_n) \cong \ZZ[x_1,x_2,\dots,x_n] /\ILambda_n
\ee
and the connective $K$-theory ring as 
\be \label{eq:CK-Fl}
\CK(\Fl_n) \cong \ZZ[\beta][x_1,x_2,\dots,x_n] / \ILambda_n[\beta].
\ee
where $\beta,x_1,x_2,\dots$ are commuting indeterminates and $ \ILambda_n \subseteq \ZZ[x_1,x_2,\dots,x_n]$ is the ideal
generated by symmetric polynomials without constant term in the variables $x_1,x_2,\dots,x_n$; see \S \ref{gr-sect}.

Let $S_n$ denote the symmetric group of permutations of $[n] := \{1,2,\dots,n\}$
and identify $w \in S_n$ with the permutation matrix in $\GL_n$ with 
$1$ in position $(i, w(i))$. 
It follows by elementary linear algebra that
the opposite Borel subgroup $B^+$ of upper triangular matrices in $\GL_n$ acts on $\Fl_n$ on the right with $n!=|S_n|$ distinct orbits.
The orbit closures $X_w := \overline{B w B^{+}} $ for $w \in S_n$ 
are the \emph{Schubert varieties} in $\Fl_n$
and one is interested in  describing the classes $[X_w] \in \CK(\Fl_n)$.

For $v \in S_n$ and $w \in S_m$, let $v \times w \in S_{n+m}$ be the permutation with $i \mapsto v(i)$ for $i \in [n]$ and $n+i \mapsto n+w(i)$ for $i \in [m]$. We also write $w^m$ for the $m$-fold product $w \times w\times \cdots \times w$, so that $1^m$ is the identity in  $S_m$. Many different polynomials correspond to each class $[X_w] \in \CK(\Fl_n)$ under the isomorphism \eqref{eq:CK-Fl}, but if one also requires a certain compatibility condition with respect to the maps $w \mapsto w \times 1^m$, 
then there is a unique family of such polynomials:

\begin{theorem}
\label{intro-thm1}
There are unique polynomials
$ \fkG_w\in \ZZ[\beta][x_1,x_2,\dots]$ for $n \in \PP$ and $w \in S_n$
such that
$ \fkG_{w} +  \ILambda_n[\beta] = [X_w] \in \CK(\Fl_n)$
and $\fkG_{w} = \fkG_{w \times 1}$.
\end{theorem}

 This statement combines several known results reviewed in Section~\ref{gr-sect}.
The polynomials $\fkG_w$ are the \emph{(generalized) Grothendieck polynomials} introduced in \cite{FK1994}.
The \emph{Schubert polynomials} (see \cite[Chapter 2]{Manivel})
are the special case of these functions with $\beta=0$.
Setting $\beta =-1$ and replacing each variable $x_i$ by $1-x_i$,
alternatively, recovers Lascoux and Sch\"utzenberger's original definition of Grothendieck polynomials
in \cite{Lascoux1990,LS1983}.

It is a remarkable observation of Fomin and Kirillov \cite{FK1994} that the sequence of polynomials $\fkG_{1^m \times w}$ converges as $m \to \infty$ to a symmetric function:

\begin{theorem}[{\cite
[Theorem 2.3]
{FK1994}}]
\label{intro-thm2}
There are unique symmetric functions 
$ G_w $ for each $n \in \PP$ and $w \in S_n$
such that
$G_w(x_1,\dots,x_n) = \fkG_{1^N\times w}(x_1,\dots,x_n)$
for all $N\geq n$.
\end{theorem}

Following established practice,
we refer to the symmetric functions $G_w$ as \emph{stable Grothendieck polynomials}.
These power series have a number of other interesting properties and are studied in \cite{Buch2002,BKSTY,FK1994}.

The preceding results have interesting counterparts for the 
orbit closures of the orthogonal and symplectic groups acting on $\Fl_n$.
 These actions are particularly natural to consider:
 they both have finitely many orbits, 
 and correspond to two of the three families of type A symmetric varieties \cite{RichSpring}.
 (The third family comes from the action of $\GL_p \times \GL_{n-p}$ on $\Fl_n$;
 $K$-theory representatives for the relevant orbit closures are 
studied in \cite{WyserYong2}.)

Fix nondegenerate symmetric and skew-symmetric bilinear forms on $\CC^n$.
We define the \emph{orthogonal group} $\O_n$ and the \emph{symplectic group} $\Sp_n$ as 
the subgroups of $\GL_n$ preserving these forms. Note that $n$ must be even in the skew-symmetric case. As explained in \cite[\S10]{RichSpring}, the $\O_n$-orbits on $\Fl_n$ are in bijection with the set of involutions 
\[I_n:=\{w \in S_n : w=w^{-1}\}\] while the $\Sp_n$-orbits are in bijection with the set of fixed-point-free involutions \[
\ISp_n := \{ z \in I_n : z(i)\neq i\text{ for all }i \in [n]\}. \]
We write $\{\iX_z : z \in I_n\}$ and $\{\fX_z : z \in \ISp_n\}$ for the respective families of $\K_n$-orbit closures, where $\K$ is one of the symbols $\O$ or $\Sp$; see \S \ref{matrix-sect} for explicit descriptions of these varieties.

We can now state symplectic and orthogonal analogues of Theorem~\ref{intro-thm1}:

   \begin{theorem}
   [Wyser and Yong \cite{WyserYong}]
     \label{thm:intro-fX-polynomials}
    There are unique polynomials
    \[\Gfpf_z \in \ZZ[\beta][x_1,x_2,\dots]\qquad\text{for $n \in 2\PP$ and $z \in \ISp_n$}\]
    such that
    $ \Gfpf_z +  \ILambda_n[\beta] = [\fX_z] \in \CK(\Fl_n)$
    and $\Gfpf_z = \Gfpf_{z \times 21}$.
    \end{theorem}
    
    The derivation of this statement from the results in \cite{WyserYong},
    which is not entirely trivial, is explained in Section~\ref{sp-case-sect}.
   The following theorem is new:
    
\begin{theorem} \label{thm:intro-iX-polynomials}
 There are unique polynomials
    \[ \iG_z \in \ZZ[\beta][x_1,x_2,\dots]\qquad\text{for $n \in \PP$ and $z \in I_n$}\]
    such that
    $ \iG_z +  \ILambda_n[\beta] = [\iX_z] \in \CK(\Fl_n)$
    and $\iG_z = \iG_{z \times 1}$.
   \end{theorem}
   
   We refer to 
      $\Gfpf_z$ and $\iG_z$ as \emph{symplectic} and \emph{orthogonal  Grothendieck polynomials}. 
      Setting $\beta=0$ transforms these functions to the \emph{(fixed-point-free) involution Schubert polynomials} $\iS^\fpf_z$ and $\iS_z$ studied in \cite{HMP1,HMP3,HMP5,WyserYong}.
The latter represent the cohomology classes of the orbit closures $\fX_z$ and $\iX_z$. 

      Wyser and Yong
      \cite{WyserYong} give a recursive method for computing $\Gfpf_z$ involving divided difference operators; see Theorem~\ref{sp-thm1}.
       By contrast, no simple algebraic formulas for computing $\iG_z$ are known for general $z \in I_n$.
       This notably differs from the situation for the involution Schubert polynomials $\iS_z$,
       which can again be characterized using divided differences \cite{WyserYong}.

We prove Theorems~\ref{thm:intro-fX-polynomials} and \ref{thm:intro-iX-polynomials} in a uniform way
in Section~\ref{stab-sect}
 by adapting an idea of Knutson and Miller \cite{KnutsonMiller}. The $B^+$-orbits on $B\backslash \GL_n = \Fl_n$ are naturally in bijection with the $B \times B^+$-orbits on $\GL_n$. The closures $M_w$ of the latter orbits in the space $\Mat_n$ of $n \times n$ matrices are known as \emph{matrix Schubert varieties}. 

Let $T$ denote the torus of invertible diagonal matrices in $\GL_n$.
Knutson and Miller prove that the class $[M_{w}]_T$ in the $T$-equivariant $K$-theory ring $K_T(\Mat_n)\cong\ZZ[x_1, x_2,\ldots, x_n]$
is the polynomial obtained from  $\fkG_{w}$ by setting $\beta=-1$.
Using this fact, one can show that $\fkG_{w} + I\Lambda_n[\beta] =  [X_w] \in \CK(\Fl_n)$.

The $\K_n$-orbits on $\Fl_n$ are in bijection with the $B$-orbits on $\GL_n/\K_n$. Embedding $\GL_n/\K_n$ as an open dense subset of the space of symmetric matrices $\MO_n$ or skew-symmetric matrices $\MSp_n$, as appropriate, and taking the closures of these $B$-orbits gives a family of \emph{(skew-)symmetric matrix Schubert varieties} $\KMX_z$; see Definition~\ref{defn:inv-matrix-schubert}.

The following is a consequence of our results in Section~\ref{equi-sect}:

\begin{theorem}
Fix $\K \in \{\O,\Sp\}$ and let $z \in I_n$.
Assume $n$ is even and $z \in \ISp_n$ if $\K = \Sp$. The $T$-equivariant class $[\KMX_z]_T \in K_T(\MK_n) \cong \ZZ[x_1,x_2, \ldots, x_n]$ is then the polynomial 
obtained from $\KG_z$ by setting $\beta=-1$.
\end{theorem}

We now turn to analogues of Theorem~\ref{intro-thm2}. 
The next result follows from Theorem~\ref{intro-thm2} and Corollary~\ref{gsp-cor}:

\begin{theorem} \label{intro-thm4-sp}
        There are unique symmetric functions 
$ \GSp_z$ for each $n \in 2\PP$ and $z \in \ISp_n$
such that
$\GSp_z(x_1,\dots,x_n) = \Gfpf_{(21)^N\times z}(x_1,\dots,x_n)$
for all $N\geq n$.
\end{theorem}

In the orthogonal case, we have only succeeded in proving a partial analogue of Theorem~\ref{intro-thm2}. A permutation is \emph{vexillary} if it avoids the pattern $2143$.
The following is a corollary of Theorem~\ref{final-thm}:

\begin{theorem} \label{intro-thm4-o}        
     There are unique symmetric functions 
$ \GO_z$ for each $n \in \PP$ and vexillary $z \in I_n$
such that
$\GO_z(x_1,\dots,x_n) = \iG_{1^N\times z}(x_1,\dots,x_n)$
for all $N\geq n$.
\end{theorem}
Our proof of Theorem~\ref{intro-thm4-o} relies on an explicit Pfaffian formula for $\iG_z$ when $z$ is vexillary, which we derive by realizing $\iX_z$ as a type C Grassmannian degeneracy locus and applying a formula of Anderson \cite{Anderson2017} for the $K$-theory classes of such loci.

The main result of \cite{Mar} shows that $\GSp_z$ is an $\NN[\beta]$-linear combination of the \emph{K-theoretic Schur P-functions} $\GP_\lambda$ of Ikeda and Naruse \cite{IkedaNaruse}. When $z$ is vexillary, we prove that $\GO_z$ is likewise a  \emph{$K$-theoretic Schur Q-function} $\GQ_{\lambda}$, also introduced in \cite{IkedaNaruse}. 
The degeneracy locus formulas in \cite{Anderson2017} are very complicated
and we find it amazing that the expressions we derive for $\GO_z$ in the vexillary case
coincide exactly with symmetric functions already considered in the literature.

We expect that Theorem~\ref{intro-thm4-o} holds for all involutions $z \in I_n$, and that the resulting power series $\GO_z$ are $\NN[\beta]$-linear combinations of the $\GQ_\lambda$ functions.
Section~\ref{last-sect} discusses several other related open problems.

\subsection*{Acknowledgements}

The first author was supported by Hong Kong RGC Grant ECS 26305218.
We are grateful to Dave Anderson, Bill Fulton, Zach Hamaker, Hiroshi Naruse, and Alex Yong
for many helpful comments.

\section{Preliminaries on connective $K$-theory}

This section provides
an expository overview of 
connective $K$-theory
and then describes a general method of constructing polynomial $K$-theory
representatives for orbit closures in the complete flag variety.

Throughout, the symbols $\beta$, $a_1$, $a_2$, $\dots$, $x_1$, $x_2$, $\dots$ denote commuting indeterminates.
We write $\NN = \{0,1,2,\dots\}$ and $\PP = \{1,2,3,\dots\}$ for the sets of nonnegative and positive integers,
and define $[n] := \{i \in \PP : i \leq n\}$ for $n \in \NN$.
Given $n \in \PP$,
let $S_n$ denote the usual symmetric group of bijections $[n]\to [n]$.
The  \emph{length} of a permutation $w$ is 
$\ell(w) := |\{ (i,j): i<j\text{ and }w(i) > w(j)\}|$.

\subsection{Connective $K$-theory}\label{ck-sect}

Let $X$ be a smooth complex variety.
Recall that the ordinary \emph{$K$-theory ring} of $X$
is 
the Grothendieck group $K(X)$ of coherent sheaves on $X$, equipped with a ring structure induced by the derived tensor product.
 The structure sheaf of any closed subscheme $Z\subseteq X$
has a class in $K(X)$ which we denote by
$[Z]_K$.

Let $K(X,c)$ be the Grothendieck group of coherent sheaves whose support has codimension at least $c \in \ZZ$, so that $K(X,c) = K(X)$ whenever $c \leq 0$ and $K(X,c) = 0$ whenever $c>\dim(X)$.
The derived tensor product leads to a ring structure on $K(X,c)$. 
The next definition
originates in \cite{Cai} but our notation follows \cite{Anderson2017,HIMN}.

\begin{definition}[{See \cite[Appendix A]{Anderson2017} or \cite[\S2.1]{HIMN}}]
The \emph{connective $K$-theory ring} of $X$ is the graded $\ZZ[\beta]$-algebra
\[\CK(X) := \bigoplus_{c \in \ZZ} \CK^c(X)\]
in which $\CK^c(X)$ is the image of 
the natural map $K(X,c) \to K(X,c-1)$, so that $\CK^c(X) = K(X)$ whenever $c\leq 0$. 
The maps $K(X,c) \to K(X,c-1)$ induce maps $\CK^c(X) \to \CK^{c-1}(X)$, and the $\ZZ[\beta]$-algebra structure on $\CK(X)$ is defined by letting $\CK^c(X) \to \CK^{c-1}(X)$ be multiplication by $-\beta$. 
\end{definition}

\begin{example}
A coherent sheaf on $X = \pt$ is a map $\pt \to \{V\}$ for some finite-dimensional
complex vector space $V$. The Grothendieck group $K(\pt) = \ZZ$ is generated by 
the sheaf $\pt \to \{\CC\}$. All sheaves on $\pt$ have 
codimension zero so $K(\pt,c) = \ZZ$ for $c\leq 0$ and $ K(\pt,c) = 0$ for $c>0$.
Identifying $\CK^c(\pt)\cong \ZZ$ for $c \leq 0$ with the free abelian group
$\ZZ\spanning\{(-\beta)^{-c}\}$ lets us write $\CK(\pt) = \ZZ[\beta]$.
\end{example}

Suppose $Z \subseteq X$ is a closed subscheme; its structure sheaf $\mathcal{O}_Z$ has support $Z$, so 
there is a corresponding class  in $ K(X,\codim(Z))$,
whose image  under the natural map $ K(X,\codim(Z)) \to K(X)$ is $[Z]_K$.
 The connective $K$-class of $Z$ is the image of the former class under the natural
 map $K(X,\codim(Z)) \to \CK^{\codim(Z)}(X)$, which we denote by $[Z]_{\CK}$. 
We drop the subscripts from $[Z]_{K}$ or $[Z]_{\CK}$ when these are clear from context.
These classes are related as follows:

\begin{proposition}
\label{k-ck-prop}
There is a $\ZZ[\beta]$-algebra morphism \[\psi : K(X)[\beta,\beta^{-1}] \to \CK(X)[\beta^{-1}]\] with
$\psi([Z]_K) = (-\beta)^{\codim(Z)} [Z]_{\CK}$ for any closed subscheme $Z \subseteq X$.\end{proposition}

    \begin{proof}
        The map $\psi$ is induced from the identity map $K(X) \to \CK^0(X)$. By definition, the image of $[Z]_{\CK} \in \CK^{\codim(Z)}(X)$ in $\CK^0(X)$ is $(-\beta)^{\codim(Z)} [Z]_{\CK}$, and this element of $\CK^0(X) = K(X)$ is just $[Z]_K$.
        \end{proof}
        
        \subsection{Grothendieck polynomials for Schubert varieties} \label{gr-sect}

Fix a positive integer $n$.
As in the introduction, write $\GL_n$ for the complex general linear group and 
 $B$ for the Borel subgroup of lower triangular matrices in $\GL_n$.
We are primarily interested in the preceding definitions applied to
the \emph{complete flag variety} $\Fl_n  := B\backslash \GL_n$.
For this choice of $X$, one can realize $K(X)$ and $\CK(X)$ as quotients of a polynomial ring. 

For each $i \in [n]$, there is a natural line bundle $L_i$ on $\Fl_n$,
whose fiber over an orbit $Bg \in \Fl_n$ 
is  the quotient $F_i / F_{i-1}$,
where $F_i$ is the subspace of $\CC^n$ spanned by the first $i$ rows of $g \in \GL_n$.
Let $\ILambda_n$ denote the ideal in $\ZZ[x_1,x_2,\dots,x_n]$
generated by the symmetric polynomials without constant term.

\begin{theorem}[{\cite[Theorem 2.6]{HornKirit}}]
\label{iso-thm}
There are isomorphisms
\[
K(\Fl_n) \xrightarrow{\sim} \ZZ[x_1,\dots,x_n] / \ILambda_n
\quand
\CK(\Fl_n) \xrightarrow{\sim} \ZZ[\beta][x_1,\dots,x_n] / \ILambda_n[\beta]
\]
mapping the first Chern class $c_1(L_i^\vee)$ of the line bundle dual to $L_i$ to $x_i$.
%for each $i \in [n]$.
\end{theorem}

From now on, we identify the rings
$K(\Fl_n) = \ZZ[x_1,x_2,\dots,x_n] / \ILambda_n$
and
$\CK(\Fl_n) = \ZZ[\beta][x_1,x_2,\dots,x_n] / \ILambda_n[\beta]$
via the preceding theorem.
For a closed subscheme $Z \subseteq \Fl_n$,
it is then natural to ask for a polynomial
whose image in these quotient rings gives $[Z]_K$ or $[Z]_{\CK}$.

This question is well-understood for the \emph{Schubert varieties} $X_w$.
Recall that these varieties are the closures of the 
double cosets $BwB^+ \subseteq \Fl_n$, where $B^+\subseteq\GL_n$ is the subgroup of upper triangular matrices and $w$ ranges over the symmetric group
$S_n$, viewed as the subgroup of permutation matrices in $\GL_n$.

Let $s_i = (i,i+1) \in S_n$ for each $i \in [n-1]$.
Given $f \in \ZZ[\beta][x_1,\dots,x_n]$,
let $s_i f$ be the polynomial formed from $f$ by interchanging $x_i$ and $x_{i+1}$,
and define 
\be
\label{divdiff-eq}
 \partial_i f := \tfrac{f -s_i f}{x_i-  x_{i+1}}\quand \varpi_i f := \partial_i( (1+\beta x_{i+1}) f) 
= -\beta f + (1+\beta x_i) \partial_i f.
\ee
We refer to $\partial_i$ and $\varpi_i$ as \emph{divided difference operators}.
Write $w_1w_2\cdots w_n$ for the permutation in $S_n$ with the formula $i \mapsto w_i$.

\begin{thmdef}[See \cite{FK1994}]
\label{groth-thmdef}
The \emph{Grothendieck polynomials} $\{ \fkG_w \}_{w \in S_n}$ are the unique family
in $\ZZ[\beta][x_1,\dots,x_n]$ with 
$\fkG_{n\cdots 321} = x_1^{n-1} x_2^{n-2} \cdots x_{n-1}^1$ 
and $\varpi_i \fkG_w = \fkG_{w s_i}$ for all $w \in S_n$ and $i \in [n-1]$ such that $w(i) > w(i+1)$. 
\end{thmdef}

It follows from the last property that $\varpi_i \fkG_w = -\beta \fkG_w$ if $w(i) < w(i+1)$.
It is also not hard to check that $\fkG_w = \fkG_{w\times 1}$ for all $w \in S_n$,
where $w \times1$ denotes the permutation in $S_{n+1}$
with $i \mapsto w(i)$ for $i \in [n]$ and $n+1 \mapsto n+1$.
Less obviously, one always has $\fkG_w \in \NN[\beta][x_1,x_2,\dots,x_n]$ 
\cite[Theorem 2.3]{FK1994}.

\begin{example}
The Grothendieck polynomials for $w \in S_3$ are 
\[
\ba
\fkG_{123} &= 1, \\
\fkG_{213} &= x_1, \\
\ea
\qquad\quad
\ba
\fkG_{132} & =x_1 + x_2 + \beta x_1x_2, \\
\fkG_{231} &= x_1x_2, 
\ea
\qquad\quad
\ba
\fkG_{312} &= x_1^2, \\
\fkG_{321} &= x_1^2 x_2.
\ea
\]
\end{example}

Work of Hudson, extending earlier results of Fulton and Lascoux, shows 
that
the polynomials $\fkG_w$ represent the Schubert classes
$[X_w]$ in connective $K$-theory.
Specifically, the following
 is the special case of \cite[Theorem 1.2]{Hudson}
 obtained
 by taking $V$ to be a trivial vector bundle of rank $n$ over $X=\pt$:

\begin{theorem}[{\cite[Theorem 1.2]{Hudson}}]
\label{ck-thm}
For each $w \in S_n$,
it holds that 
\[ \fkG_{w} +  \ILambda_n[\beta] = [X_{w}] \in \CK(\Fl_n).\]
\end{theorem}

We typically suppress the parameter $\beta$ in our notation,
but for the moment write $\fkG_w^{(\beta)} = \fkG_w$ for $w \in S_n$.
The \emph{Schubert polynomial} $\fkS_w$ of a permutation $w \in S_n$ (see \cite[Chapter 2]{Manivel})
is then $\fkG_w^{(0)}$. It follows that $\{\fkG_w\}_{w \in S_n}$ are linearly independent by \cite[Proposition 2.5.3]{Manivel}.

Some references use the term ``Grothendieck polynomial'' to refer to the polynomials $\fkG_w^{(-1)}$.
One loses no generality in setting $\beta=-1$ since one can show by downward induction on permutation length that
\be\label{beta-rel} (-\beta)^{\ell(w)}\fkG^{(\beta)}_w =  \fkG_w^{(-1)}(-\beta x_1, -\beta x_2,\dots,-\beta x_n).\ee
This lets us translate any formulas in $\fkG_w^{(-1)}$ to formulas 
in $ \fkG_w = \fkG_w^{(\beta)}$.
The specialization $\beta=-1$ is natural since it corresponds to ordinary $K$-theory:

\begin{theorem}[{\cite[Theorem 3]{FultonLascoux}}]
\label{prek-thm}
For each $w \in S_n$,
it holds that 
\[ \fkG_{w}^{(-1)} +  \ILambda_n = [X_{w}] \in K(\Fl_n).\]
        \end{theorem}
        
We can now describe the map in Proposition~\ref{k-ck-prop}
 for $X=\Fl_n$.
 
\begin{corollary}\label{fair-cor}
         If $X = \Fl_n$ then the map $\psi : K(\Fl_n)[\beta,\beta^{-1}] \to \CK(\Fl_n)[\beta^{-1}]$ in Proposition~\ref{k-ck-prop}
	is the ring homomorphism sending $x_i\mapsto -\beta x_i$ for $i \in [n]$.
\end{corollary}
                  
                  \begin{proof}
Since $\codim(X_w) = \ell(w)$, it follows from Theorems~\ref{ck-thm} and \ref{prek-thm} that
$\psi(\fkG^{(-1)}_w + \ILambda_n) = (-\beta)^{\ell(w)}\fkG^{(\beta)}_w + \ILambda_n[\beta]$.
By \eqref{beta-rel}, this agrees with the ring homomorphism $\ZZ[x_1,\dots,x_n]/\ILambda_n \to \ZZ[\beta][x_1,\dots,x_n]/\ILambda_n[\beta]$ sending $x_i\mapsto -\beta x_i$ for $i \in [n]$.
As $\{ \fkG^{(-1)}_w + \ILambda_n : w\in S_n\}$ is a basis for $K(\Fl_n)$ by
\cite[Proposition 2.5.3 and Corollary 2.5.6]{Manivel}, we conclude that $\psi$ is equal to the latter map.
    \end{proof}

\subsection{Matrix Schubert varieties}\label{matrix-sect}

For the remainder of this section, $\K$ denotes one of the symbols $\O$ or $\Sp$.
Fix $n \in \PP$ and write $I_n$ and $\ISp_n$ for the respective sets of involutions and fixed-point-free involutions
in the finite symmetric group $S_n$. If $n$ is odd then 
$\ISp_n =\varnothing$.

If $V_1,V_2$ are complex vector spaces and $\alpha : V_1 \times V_2 \to \CC$ is a bilinear form, then we let $\rank(\alpha)$ denote the rank of the map $V_2 \to V_1^*$ given by $v \mapsto \alpha(\cdot,v)$. 
Let $\alpha^{\O}_n$ be a fixed symmetric nondegenerate bilinear form on $\CC^n$.
When $n$ is even, let $\alpha^{\Sp}_n$ be a fixed skew-symmetric nondegenerate bilinear form on $\CC^n$.
Define $\O_n$ to be the subgroup of $\GL_n$ preserving $\alpha_n^{\O}$ and $\Sp_n$ the subgroup preserving $\alpha_n^{\Sp}$. Write $A_{[i][j]}$ for the upper-left $i \times j$ corner of a matrix $A$.

Given $E = Bg \in \Fl_n$ and $i \in [n]$, define $E_i\subseteq \CC^n$ to be the subspace spanned by the first $i$ rows of $g \in \GL_n$;
these spaces do not depend on the choice of $g$.

\begin{definition} \label{defn:orbit-closures}
    Given $\K \in \{ \O,\Sp\}$ and $z \in \I_n$, let
\[
    \KX_z := \{ E \in \Fl_n : \rank(\alpha_n^\K|_{E_i \times E_j}) \leq \rank(z_{[i][j]}) \text{ for $i,j \in [n]$}  \},
\]
where we identify $z$ with its permutation matrix.
\end{definition}

Each $\KX_z$ is a closed subvariety of $\Fl_n$.
The correspondence $z \mapsto \iX_z$
is a bijection from $\I_n$ to the set of closures of the $\O_n$-orbits on $\Fl_n$;
when $n$ is even, $z\mapsto\fX_z$ is likewise a bijection
from $\ISp_n$ to the set of closures of the $\Sp_n$-orbits on $\Fl_n$ \cite{Wyser}.
Although 
we are primarily interested in $\fX_z$ in the case when $z$ is fixed-point-free,
we have defined $\fX_z$ for any involution $z \in \I_{n}$ 
and this flexibility will occasionally be convenient.

Many of the rank conditions in Definition~\ref{defn:orbit-closures} are redundant. The essential rank conditions can be read off from the following diagrams.

\begin{definition}\label{rothe-def}
    Let $z \in \I_n$. The \emph{orthogonal Rothe diagram} of $z$ is 
    \begin{equation*}
        \DO(z) := \{(i, z(j)) : \text{$(i,j) \in [n] \times [n]$ and $z(i) > z(j) \leq i < j$}\}.
    \end{equation*}
    The \emph{symplectic Rothe diagram} of $z$ is 
    \begin{equation*}
        \DSp(z) := \{(i, z(j)) : \text{$(i,j) \in [n] \times [n]$ and $z(i) > z(j) < i < j$}\}.
    \end{equation*}
\end{definition}

The diagrams $\DO(z)$ and $\DSp(z)$ are the subsets of positions in the usual \emph{Rothe diagram}
$D(z) := \{(i, z(j)) : \text{$(i,j) \in [n] \times [n]$, $z(i) > z(j)$, and $i < j$}\}$
that are weakly and strictly below the main diagonal, respectively.

\begin{definition}\label{ess-def}
    The \emph{essential set} of a subset $D \subseteq \PP \times \PP$ is
    \begin{equation*}
        \Ess(D) := \{(i,j) \in D : (i,j+1), (i+1,j) \notin D\}.
    \end{equation*}
\end{definition}

\begin{proposition}[{\cite[Proposition 3.16]{HMP1}}] 
\label{ess-rpo}
    Let $\K \in \{\O,\Sp\}$ and  $z \in \I_n$. Then
\[ %\label{eq:essential-rank-conditions}
        \KX_z = \left\{ E \in \Fl_n : \rank(\alpha_n^\K|_{E_i \times E_j}) \leq \rank(z_{[i][j]})\text{ for $(i,j) \in \Ess(\DO(z))$}  \right\}.
\]
    Moreover, if $n$ is even and $z \in \ISp_n$ then
\[
        \fX_z = \left\{ E \in \Fl_n : \rank(\alpha_n^\Sp|_{E_i \times E_j}) \leq \rank(z_{[i][j]})\text{ for $(i,j) \in \Ess(\DSp(z))$}  \right\}.
\]
     \end{proposition}

\begin{example}
    Let $z = (1,3) \in \I_3$. Then
    \begin{equation*}
        \DO(z) = \{(1,1),(2,1)\} = \boxed{\begin{array}{ccc}
            \circ & \cdot & \times\\
            \circ & \times & \cdot\\
            \times & \cdot & \cdot 
        \end{array}}
    \end{equation*}
    where elements of $\DO(z)$ are drawn with $\circ$, points $(i,z(i))$ with $\times$, and the diagram is shown in matrix coordinates with $(1,1)$ at the upper left. We have $\Ess(\DO(z)) = \{(2,1)\}$ and $\rank(z_{[2][1]})= 0$, so 
    \begin{equation*}
        \iX_z = \{E \in \Fl_3 : \rank(\alpha_3^\O|_{E_1 \times E_2}) \leq 0\} = \{E \in \Fl_3 : E_2 \subseteq E_1^{\perp}\}.
    \end{equation*}
\end{example}

Let $\MO_n$ (respectively, $\MSp_n$) be the set of complex $n \times n$ matrices that are symmetric (respectively, skew-symmetric).
The space $\MK_n$ contains a family of varieties closely related to $\KX_z$:

\begin{definition} \label{defn:inv-matrix-schubert}
    Given $\K \in \{\O,\Sp\}$ and  $z \in \I_n$, let 
    \begin{equation*}
        \KMX_z := \left\{ A \in \MK_n : \text{$\rank(A_{[i][j]}) \leq \rank(z_{[i][j]})$ for $i,j \in [n]$}  \right\}.
    \end{equation*}
\end{definition}
We call the closed subvariety $\iMX_z$ (respectively, $\fMX_z$) a \emph{symmetric matrix Schubert variety}
(respectively, \emph{skew-symmetric matrix Schubert variety}). If one  allows arbitrary $z \in S_n$ and arbitrary matrices in Definition~\ref{defn:inv-matrix-schubert}, then one 
recovers Knutson and Miller's notion of a \emph{matrix Schubert variety} from \cite{KnutsonMiller}.

The variety $\KMX_z$ is also an orbit closure,
but now for the Borel subgroup $B \subseteq \GL_n$,
which acts on $A \in \MK_n$ by $b : A \mapsto bAb^\top$. The 
maps $z \mapsto \iMX_z$ and $z \mapsto \fMX_z$
are bijections from $I_n $ and $\ISp_n$ to the closures of the $B$-orbits in $\MO_n$ and $\MSp_n$,
respectively; see \cite{BagCher, Cher2011}.
There is an analogue of Proposition~\ref{ess-rpo}:

\begin{proposition} 
\label{ess-rpo2}
    Let $\K \in \{\O,\Sp\}$ and  $z \in \I_n$. Then
\[ %\label{eq:essential-rank-conditions2}
        \KMX_z = \left\{ A \in \MK_n : \rank(A_{[i][j]}) \leq \rank(z_{[i][j]})\text{ for $(i,j) \in \Ess(\DO(z))$}  \right\}.
\]
    Moreover, if $n$ is even and $z \in \ISp_n$ then
\[
        \fMX_z = \left\{  A \in \MSp_n : \rank(A_{[i][j]}) \leq \rank(z_{[i][j]})\text{ for $(i,j) \in \Ess(\DSp(z))$}  \right\}.
\]
     \end{proposition}

\begin{proof}
One can almost repeat the proof of \cite[Proposition 3.16]{HMP1} verbatim;
the argument in \cite{HMP1} goes through after 
replacing ``$y$'' by ``$z$'' and redefining ``$C_{ij}$'' to be the set of symmetric (when $\K=\O$) or skew-symmetric
(when $\K = \Sp$ and $z \in \ISp_n$) $n\times n$ matrices $A$ with $\rank(A_{[i][j]}) \leq \rank(z_{[i][j]})$.
\end{proof}

\subsection{Grothendieck polynomials for orbit closures}
\label{equi-sect}

The orthogonal and symplectic matrix Schubert varieties have 
canonical polynomial representatives in equivariant $K$-theory.
Here, we use these polynomials to give a uniform definition 
of the Grothendieck polynomials $\iG_z$ and $\Gfpf_z$
described in the introduction.

Suppose $G$ is a linear algebraic group acting on a smooth complex variety $X$. The \emph{$G$-equivariant $K$-theory ring} $K_G(X)$ is the Grothendieck group of $G$-equivariant vector bundles on $X$ with tensor product as multiplication, or equivalently the Grothendieck group of $G$-equivariant coherent sheaves on $X$ (now with the derived tensor product as multiplication). 
If $Z \subseteq X$ is a $G$-invariant subscheme, then we write $[Z] \in K_G(X)$ for the class of its structure
sheaf.

A $G$-equivariant vector bundle over a point is just a representation of $G$, so $K_G(\textrm{pt})$ is the Grothendieck ring $R(G)$ of finite-dimensional complex rational representations of $G$. If $T \cong (\CC^{\times})^n$ is a torus then $K_T(\textrm{pt})$ can be identified with 
the ring $\ZZ[a_1^{\pm 1},a_2^{\pm 1} \ldots, a_n^{\pm 1}]$; the one-dimensional representation on which $(t_1,t_2 \ldots, t_n) \in T$ acts as multiplication by $t_1^{m_1} t_2^{m_2}\cdots t_n^{m_n}$ has class $a_1^{m_1} a_2^{m_2}\cdots a_n^{m_n}$, and every $T$-representation is a direct sum of such one-dimensional representations.

We summarize a few other properties we will need from \cite[\S 5.2]{chriss-ginzburg}:

\begin{itemize}
%    \item A $G$-equivariant vector bundle over a point is just a representation of $G$, so $K_G(\textrm{pt})$ is the Grothendieck ring $R(G)$ of finite-dimensional complex rational representations of $G$. In particular, if $T \cong (\CC^{\times})^n$ is a torus then $K_T(\textrm{pt})$ can be identified with $\ZZ[a_1^{\pm 1}, \ldots, a_n^{\pm 1}]$: the one-dimensional representation on which $(t_1, \ldots, t_n) \in T$ acts as multiplication by $t_1^{m_1} \cdots t_n^{m_n}$ has class $a_1^{m_1} \cdots a_n^{m_n}$, and every $T$-representation is a direct sum of such one-dimensional representations.
    
    \item A $G$-equivariant map $f : X \to Y$ between smooth complex varieties defines a pullback $f^* : K_G(Y) \to K_G(X)$, and this assignment is functorial. The pullback of $X \to \pt$ makes the ring $K_G(X)$ into an algebra over $K_G(\pt) \cong R(G)$. The pullbacks $f^*$ are $R(G)$-algebra homomorphisms.
    
    \item If $f : X \to Y$ is a flat morphism (e.g., the projection of a fiber bundle or inclusion of an open subset), then $f^*([Z]) = [f^{-1}(Z)]$ for any $G$-invariant subscheme $Z$. Here, $f^{-1}(Z)$ is the scheme-theoretic inverse image, but if $Z$ and the fibers of $f$ are reduced, then the flatness of $f$ implies $f^{-1}(Z)$ reduced \cite[Proposition 11.3.13]{EGA4}. This will always be the case for us, so we can take $f^{-1}(Z)$ to be the set-theoretic inverse image.

    \item If $V$ is a finite-dimensional linear representation of $G$, then 
    there are isomorphisms $K_G(V) \cong K_G(\textrm{pt}) \cong R(G)$ \cite[Corollary 5.4.21]{chriss-ginzburg}.
    %If $f$ is the inclusion of a smooth subvariety $X \hookrightarrow Y$, then
    %\begin{equation} \label{eq:pullback}
    %    f^*([Z]) = \sum_{j=0} (-1)^j [\Tor^{\mathcal{O}_Y}_j(\mathcal{O}_X, \mathcal{O}_Z)].
    %\end{equation}
    %The $j = 0$ term of this sum is $[f^{-1}(Z)]$, the class of the scheme-theoretic inverse image $f^{-1}(Z)$.
    
    \item Given a group homomorphism $\phi : H \to G$, there is a ring homomorphism $K_G(X) \to K_H(X)$ sending $[Z]$ to $[Z]$, since one can view a $G$-equivariant vector bundle as $H$-equivariant via $\phi$. In particular, taking $H$ to be the trivial subgroup of $G$, there is such a map $K_G(X) \to K(X)$.

    \item If $G$ acts freely on $X$, then the pullback of the quotient $X \to X/G$ defines an isomorphism $K(X/G) \xrightarrow{\sim} K_G(X)$.
    
    %\item For a closed subgroup $H \subseteq G$, define $G \times_H X$ as the quotient of $G \times X$ by the $H$-action $h \cdot (g,x) = (gh, h^{-1}x)$. This set can be given a variety structure. There is an isomorphism $\Ind_H^G : K_H(X) \to K_G(G \times_H X)$ with the property that if $V \subseteq X$ is an $H$-invariant subscheme, then $\Ind_H^G [V] = [G \times_H V]$.
\end{itemize}

For each symbol $\K \in \{\O,\Sp\}$, 
fix an $n\times n$ matrix $\Omega_n^\K$ with $\alpha_n^\K(v,w) = v^\top \Omega_n^\K w$ for all $v, w \in \CC^n$. 
Define $T$ to be the torus of invertible diagonal matrices in $\GL_n$. Each $t \in T$ acts on $\GL_n$ by left multiplication and on $A \in \MK_n$ by $t : A \mapsto tAt$. 
Let $\sigma^\K_n : \GL_n \to \MK_n$ be the $T$-equivariant map with $\sigma^\K_n(g) := g \Omega_n^\K g^\top$
and write 
\[ (\sigma^\K_n)^* : K_T(\MK_n) \to K_T(\GL_n)\]
for its pullback.

If $S \subseteq \Fl_n = B \backslash \GL_n$, then we let $B\cdot S:=\{g \in \GL_n : Bg \in S\}$.
The pullback of the quotient $\GL_n \to T \backslash \GL_n$ is an isomorphism 
$K(T \backslash \GL_n) \xrightarrow{\sim} K_T(\GL_n)$ since $T$ acts freely on $\GL_n$.
The forgetful map 
$K_B(\GL_n) \to K_T(\GL_n)$
is also an isomorphism \cite[\S 5.2.18]{chriss-ginzburg}. 
Composing these maps
gives an isomorphism 
\be\label{ktkt-eq}
K(\Fl_n) = K(B\backslash \GL_n) \xrightarrow{\sim} K_T(\GL_n)
\ee
sending $[Z]$ to $[B \cdot Z]$.
Let $\phi : K_T(\GL_n) \xrightarrow{\sim} K(\Fl_n)$ be the inverse of this map.

\begin{theorem}\label{thm:same-pol}
Choose a symbol $\K \in \{\O, \Sp\}$ and assume $n$ is even if $\K = \Sp$.
The composition 
$K_T(\MK_n) \xrightarrow{(\sigma^\K_n)^*}  K_T(\GL_n) \xrightarrow{\phi} K(\Fl_n)$
maps $[\KMX_z]\mapsto [\KX_z]$
for each $z \in I_n$.
\end{theorem}

\begin{proof}
We just need to show that $(\sigma^\K_n)^*$ maps $[\KMX_z]\mapsto [B\cdot \KX_z]$ for any $z \in \I_n$.
Write $\K_n$ for the group $\O_n$ or $\Sp_n$ corresponding to the symbol $\K$.
The map $\sigma^\K_n$ factors as the quotient map $\GL_n \to \GL_n/\K_n$ followed by the map $\GL_n/\K_n \to \MK_n$ sending $g\K_n\mapsto g\Omega_n^\K g^\top$, which is an isomorphism from $\GL_n/\K_n$ onto the open subset of invertible matrices in $\MK_n$. This implies that $\sigma^\K_n$ is a flat morphism, because it is the composition of two flat morphisms: the projection of a fiber bundle and the inclusion of an open subset.

It now suffices to show that $(\sigma^\K_n)^{-1}(\KMX_z) = B\cdot \KX_z$. Indeed, we have 
$g \in B \cdot \KX_z$ if and only if the rows $g_1, g_2, \dots, g_n$ of $g$ are such that the matrix $A = [\alpha_n^\K(g_p,g_q)]_{p,q \in [n]}$ has $\rank(A_{[i][j]}) \leq \rank(z_{[i][j]})$ for any $i,j \in [n]$. But $A = \sigma^\K_n(g)$, so this condition is equivalent to $\sigma^\K_n(g) \in \KMX_z$.
\end{proof}

One way to realize the
isomorphism $K(\Fl_n) \cong \ZZ[x_1,x_2, \ldots, x_n]/\ILambda_n$
in Theorem~\ref{iso-thm}
 is as follows. 
Let $\Mat_n$ denote the algebra of complex $n\times n$ matrices.
Since $\Mat_n$ is a finite-dimensional representation of $T$
under the action $t : A \mapsto tA$,
the equivariant $K$-theory ring $K_T(\Mat_n)$ is the representation ring $R(T) \cong \ZZ[a_1^{\pm1},a_2^{\pm1} \ldots, a_n^{\pm1}]$, and the following diagram commutes:
  \[
      \begin{tikzcd}
      K_T(\Mat_n) \arrow[d, "\iota^*"] \arrow[r, "\sim"] &     \ZZ[a_1^{\pm 1}, a_2^{\pm1},\ldots, a_n^{\pm 1}] \arrow[d] \\
      \hspace{-1.7cm}K(\Fl_n) \cong K_T(\GL_n) \arrow[r, "\sim"]  &{\displaystyle \frac{\ZZ[a_1^{\pm 1}, a_2^{\pm1},\ldots, a_n^{\pm 1}]}{\langle e_d(a_1,a_2, \ldots, a_n) - {n \choose d} : d \in [n]\rangle}}
      \end{tikzcd}
  \]
Here, the vertical map on the left is the pullback of the inclusion $\iota : \GL_n \hookrightarrow \Mat_n$. Sending $a_i \mapsto 1-x_i$ gives an isomorphism from the ring in the lower right to $\ZZ[x_1, x_2,\ldots, x_n]/\ILambda_n$. This change of variables reflects a general relationship between $K(X)$ and the Chow ring of $X$.

Now suppose $Y \subseteq \Mat_n$ and $Z \subseteq \Fl_n$ are closed subschemes 
such that $\iota^*[Y] := [\iota^{-1}(Y)] = [Z] \in K(\Fl_n) $.
The class  $[Y] \in K_T(\Mat_n)$
may be canonically identified with a Laurent polynomial in $\ZZ[a_1^{\pm 1}, \ldots, a_n^{\pm 1}]$
via the diagram above, and it can be shown that this element is actually a polynomial in $a_1,\ldots,a_n$
\cite[\S6.6]{chriss-ginzburg}.
After applying the change of variables $a_i \mapsto 1-x_i$,
this polynomial becomes a representative for $[Z]$ in 
the quotient $\ZZ[x_1,\ldots,x_n]/\ILambda_n \cong K(\Fl_n)$.

If $V$ is any finite-dimensional linear representation of $T$
and $\sigma :  \Mat_n \to V$ is a $T$-equivariant map, 
then composing the pullback $\sigma^* : K_T(V) \to K_T(\Mat_n)$ with the isomorphism 
$K_T(\Mat_n) \cong R(T)$ coincides with the canonical isomorphism $K_T(V) \cong R(T)$
described by \cite[Corollary 5.4.21]{chriss-ginzburg}.
Therefore, taking $\sigma$ to be the map $\Mat_n \to \MK_n$ with $g \mapsto g \Omega_n^\K g^\top$,
we can repeat everything in the previous paragraph for 
closed subschemes $Y \subseteq \MK_n$.
In particular, to obtain ``canonical'' polynomial representatives for   
the varieties $Z = \KX_z$, we 
can apply the preceding construction with $Y = \KMX_z$:

\begin{definition}\label{main-def} For each $\K \in \{\O,\Sp\}$ and $z \in \I_n$, let 
\[\KG_z \in \ZZ[\beta,\beta^{-1}][x_1,x_2,\dots,x_n]\] be the polynomial
obtained from 
$[\KMX_z]  \in K_T(\MK_n) \cong \ZZ[a_1^{\pm 1}, \ldots, a_n^{\pm 1}] $
by substituting $a_i \mapsto 1 + \beta x_i$  for $i \in [n]$ and then dividing by $(-\beta)^{\codim(\KMX_z)}$.
\end{definition}

After the substitution $a_i \mapsto 1 + \beta x_i$, the lowest degree of a monomial appearing in $[\KMX_z]$ is $\codim(\KMX_z)$ by \cite[Claim 8.54]{MillerSturmfels}, so $\KG_z$ is in fact completely determined by $[\KMX_z]$. These codimensions can also be computed combinatorially: if $y \in I_n$ and $z \in \ISp_n$ then
$\codim(\iMX_y) = |\DO(y)|$ and $\codim(\fMX_z) = |\DSp(z)|$ \cite[Lemma 5.4]{Pawlowski}.
A method for computing $\codim(\fMX_z)$ for $z \in I_n\setminus \ISp_n$ is 
implicit in the proof of 
\cite[Theorem 6.11]{Pawlowski}, though this is slightly nontrivial.
We refer to $\iG_z$ and $\Gfpf_z$ as \emph{orthogonal} and \emph{symplectic Grothendieck polynomials}.

The polynomials $\KG_z$ can actually be defined without inverting $\beta$:

\begin{theorem}\label{welldef-thm}
For each $\K \in \{\O,\Sp\}$ and $z \in \I_n$, it holds that 
\[ \KG_z \in \ZZ[\beta][x_1,x_2,\dots,x_n]
\qquand
\KG_z + \ILambda_n[\beta]  = [\KX_z] \in \CK(\Fl_n).\]
\end{theorem}

\begin{proof}
By the preceding discussion and Theorem~\ref{thm:same-pol},
applying the change of variables $a_i \mapsto 1-x_i$ to 
$[\KMX_z]  \in K_T(\MK_n) \cong \ZZ[a_1^{\pm 1}, \ldots, a_n^{\pm 1}] $
gives a polynomial in $\ZZ[x_1,\ldots,x_n]$ whose image in 
$K(\Fl_n) = \ZZ[x_1,\ldots,x_n]/\ILambda_n$
is $[\KX_z]$.
Since one obtains $(-\beta)^{\codim(\KMX_z)} \KG_z$ from this polynomial by 
substituting $x_i \mapsto -\beta x_i$ by Corollary~\ref{fair-cor},
Proposition~\ref{k-ck-prop} 
implies
that we have
$\KG_z + \ILambda_n[\beta,\beta^{-1}]  = [\KX_z] \in \CK(\Fl_n)[\beta^{-1}]$.

To finish the proof,
it is enough to show that 
after substituting $a_i \mapsto 1-x_i$, the polynomial $[\KMX_z]$ has no
terms of degree less than $\codim(\KMX_z)$ in the $x_i$ variables.
However, as will be explained in more detail in Section~\ref{dom-sect},
this polynomial can be computed in terms of multigraded Hilbert series,
and from this perspective the needed degree property is exactly \cite[Claim 8.54]{MillerSturmfels}.
\end{proof}

\begin{example}\label{sp-ex1}
The symplectic Grothendieck polynomials for $z \in \ISp_4$ are 
\[
\ba
\Gfpf_{2143} & = 1,\\
\Gfpf_{3412} &= x_1 + x_2 + \beta x_1x_2, \\
\Gfpf_{4321} &= x_1^2 + x_1 x_2 + x_1 x_3 + x_2 x_3 +  2\beta x_1 x_2 x_3 + \beta x_1^2 x_2 + \beta x_1^2 x_3 + \beta^2 x_1^2 x_2 x_3.
\ea
\]
The smallest example of $\Gfpf_z$ where $z$ is not $\Sp$-dominant
(see Theorem~\ref{dom-thm})
 is
\[
\ba
\Gfpf_{215634} =\ 
&x_1 + x_2 + x_3 + x_4 + \beta x_1 x_2 + \beta x_1 x_3 + \beta x_1 x_4 + \beta x_2 x_3 + \beta x_2 x_4 \\& + \beta x_3 x_4 + \beta^2 x_1 x_2 x_3 + \beta^2 x_1 x_2 x_4 + \beta^2 x_1 x_3 x_4 + \beta^2 x_2 x_3 x_4 \\&+ \beta^3 x_1 x_2 x_3 x_4.
\ea
\]
We have computed these examples using  Theorem~\ref{sp-thm1}.
\end{example}

  \begin{example}\label{o-ex}
  The orthogonal Grothendieck polynomials for $z \in I_3$ are 
\[
\ba
\iG_{123} & = 1, \\
\iG_{213} & = 2x_1 + \beta x_1^2, \\
\iG_{132} & = 2 x_1+2 x_2 + \beta x_1^2 + 4 \beta x_1 x_2 + \beta x_2^2+2\beta^2  x_1^2 x_2+2\beta^2  x_1 x_2^2 + \beta ^3x_1^2 x_2^2, \\
\iG_{321} & = 2 x_1^2+2 x_1 x_2+ \beta x_1^3 + 3\beta  x_1^2 x_2+ \beta^2 x_1^3 x_2.
\ea
\]
We have computed these examples using Theorem~\ref{cone-thm} and \textsf{Macaulay2}.
  \end{example}

\section{More on Grothendieck polynomials}
\label{sp-sect}

Continue to let $n$ be a fixed positive integer.
Our goal in this section is to outline the notable properties of
the orthogonal and symplectic 
Grothendieck polynomials
$\iG_z$ and $\Gfpf_z$.
The results here will also explain more direct methods 
of computing these polynomials.

\subsection{Stability}\label{stab-sect}

To start, we prove that the polynomials $\KG_z$ for $\K \in \{ \O,\Sp\}$ are  
stable under the natural inclusions $I_n \hookrightarrow I_{n+1}$ and $\ISp_n \hookrightarrow \ISp_{n+2}$ (applied to the indices $z$). 
In the $\K=\Sp$ case, this corresponds to \cite[Theorem 4]{WyserYong}.

Define a map $p : \MK_{n+1} \twoheadrightarrow \MK_{n}$ by $p(A) = A_{[n][n]}$. 
To distinguish between the tori in $\GL_n$ and $\GL_{n+1}$, write $T_n =T$
for the subgroup of  invertible diagonal matrices in $\GL_n$.
Letting the last factor of $T_{n+1}$ act on $\MK_n$ trivially, the map $p$ is then $T_{n+1}$-equivariant,
and the projection 
$T_{n+1} \to T_n$ induces a ring homomorphism 
$K_{T_{n}}(\MK_{n}) \to K_{T_{n+1}}(\MK_{n})$ with $[Z]\mapsto [Z]$.

\begin{lemma} \label{lem:class-stability}  
Choose a symbol $\K \in \{\O, \Sp\}$.
The composition 
\be\label{compo-eq}
K_{T_{n}}(\MK_{n}) \to K_{T_{n+1}}(\MK_{n}) \xrightarrow{p^*} K_{T_{n+1}}(\MK_{n+1})
\ee
maps $[\KMX_{z}]\mapsto [\KMX_{z \times 1}]$ for each $z \in \I_n$. 
\end{lemma}

\begin{proof}
Since $\Ess(\DK(z)) = \Ess(\DK(z \times 1))$, 
it follows in view of Proposition~\ref{ess-rpo2} that $p^{-1}(\KMX_{z}) = \KMX_{z \times 1}$, which suffices as $p^*[\KMX_z] = [p^{-1}(\KMX_{z})]$.
\end{proof}

Recall that if $w \in S_n$ then we write $w\times 21$ for the permutation
in $S_{n+2}$ that maps $i \mapsto w(i)$ for $i \in [n]$, $n+1 \mapsto n+2$, and $n+2\mapsto n+1$.

\begin{theorem} \label{thm:poly-stability} 
For each $\K \in \{\O, \Sp\}$ and $z \in \I_n$ it holds that $\KG_{z \times 1} = \KG_z$.
Moreover, if $n$ is even and $z \in \ISp_n$ then $\Gfpf_{z\times 21} = \Gfpf_z$.
 \end{theorem}
 
\begin{proof}
The composition \eqref{compo-eq} can be identified with
\begin{equation} \label{eq:RT-comp}
R(T_n) \to R(T_{n+1}) \xrightarrow{\id} R(T_{n+1}),
\end{equation}
where the first arrow is the linear map that sends each representation $\pi : T_n \to \GL(V)$ to $\pi \circ p : T_{n+1} \to \GL(V)$; the second arrow must be the identity map since this is the unique $R(T_{n+1})$-algebra morphism
 $R(T_{n+1}) \to R(T_{n+1})$.
After identifying $R(T_n)$ with $\ZZ[a_1^{\pm 1}, \ldots, a_n^{\pm 1}]$, \eqref{eq:RT-comp} becomes the inclusion
\begin{equation*}
\ZZ[a_1^{\pm 1}, \ldots, a_n^{\pm 1}] \hookrightarrow \ZZ[a_1^{\pm 1}, \ldots, a_{n+1}^{\pm 1}] \xrightarrow{\id} \ZZ[a_1^{\pm 1}, \ldots, a_{n+1}^{\pm 1}],
\end{equation*}
so the first claim in the theorem follows from Lemma~\ref{lem:class-stability}.

For the second claim, assume $n$ is even and $z \in \ISp_n$. If $u = z \times 21$ and $v=z\times 1^2$,
then we have $\rank(u_{[n+1][n+1]}) +1=\rank(v_{[n+1][n+1]}) =n+1$
while
$\rank(u_{[i][j]}) = \rank(v_{[i][j]})$ for all $(n+1,n+1)\neq (i,j) \in [n+2]\times [n+2]$.
Since $\rank(A_{[n+1][n+1]})$ is necessarily even 
if $A$ is skew-symmetric,
it follows by Definition~\ref{defn:inv-matrix-schubert} that
 $\fMX_{z\times 21}= \fMX_{z\times 1^2}$, so $\Gfpf_{z\times 21} =\Gfpf_{z \times 1^2} = \Gfpf_{z\times 1} =\Gfpf_z$.
\end{proof}

As an application, we can now prove Theorems~\ref{thm:intro-fX-polynomials} and \ref{thm:intro-iX-polynomials}.
We require one lemma.
Recall that $\ILambda_n$ is the ideal in $\ZZ[x_1,x_2,\dots,x_n]$ generated by the elements 
that are symmetric in $x_1,x_2,\dots,x_n$ and have zero constant term. 

\begin{lemma}\label{ilam-lem}
Suppose $n_1,n_2,n_3,\dots$ is a sequence of positive integers with $\lim_{i\to \infty} n_i = \infty$.
Then $\bigcap_{i=1}^\infty \ILambda_{n_i} = 0$.
\end{lemma}

The following argument is similar to the proof of \cite[Lemma 2.11]{Pawlowski}.
Write $\{\fkS_w\}_{w \in S_n}$ for the usual family of Schubert polynomials
 (see \cite[Chapter 2]{Manivel}).
 
\begin{proof}
    Suppose $f$ is a nonzero polynomial. Then $f$ is a nontrivial linear combination of elements of $\{ \fkS_w : w \in S_N\}$
for some $N=n_i$ \cite[Proposition 2.5.4]{Manivel}. As $\{\fkS_w + \ILambda_N :w \in S_N\}$ is a $\ZZ$-basis for $\ZZ[x_1,\dots,x_N]/\ILambda_N$
\cite[Proposition 2.5.3 and Corollary 2.5.6]{Manivel}, $f$ is therefore nonzero in this ring.
\end{proof}

\begin{proof}[Proof of Theorems~\ref{thm:intro-fX-polynomials} and \ref{thm:intro-iX-polynomials}]
The existence assertions in these results
are Theorems~\ref{welldef-thm} and \ref{thm:poly-stability}.
The uniqueness of $\iG_z$ and $\Gfpf_z$
follows from Lemma~\ref{ilam-lem},
which
implies that 
$\bigcap_{i=1}^\infty \ILambda_{n+i}[\beta] = \bigcap_{i=1}^\infty \ILambda_{n+2i}[\beta] =0$
for any $n \in \PP$.
\end{proof}

\subsection{Dominant formulas}\label{dom-sect}

Continue to let 
$T=T_n$ be the torus of invertible diagonal matrices in $\GL_n$.
When $V$ is a rational representation of $T$ and $Z\subseteq V$ is a $T$-invariant subscheme, there is a useful algebraic method for computing the polynomial $[Z] \in K_T(V)$,
which we will use to derive an explicit product formula for certain instances of the polynomials $\KG_z$.

 Let $X(T) = \Hom(T, \CC^\times)$ be the character group of $T$. For $\lambda \in X(T)$, let 
\[V_{\lambda} = \{v \in V : \text{$tv = \lambda(t)v$ for $t \in T$}\}\subseteq V\] be the $\lambda$-weight space of $V$. Choosing coordinates on $T$ uniquely identifies integers $m_1, \ldots, m_n$ with $\lambda(t) = t_1^{m_1} \cdots t_n^{m_n}$ for all $t = (t_1, \ldots, t_n) \in T$. Accordingly, we identify $\lambda$ with $(m_1, \ldots, m_n)$, and write $a^{\lambda}$ for the monomial $a_1^{m_1} \cdots a_n^{m_n}$.

\begin{definition} Suppose $V$ is a rational representation of $T$ such that each weight space $V_{\lambda}$ is finite-dimensional. The \emph{Hilbert series} of $V$ is then
\[\Hs(V,a) := \sum_{\lambda \in X(T)} \dim(V_{\lambda}) a^{\lambda}.\] 
When the variables are clear from context, we write $\Hs(V)$ in place of $\Hs(V,a)$.
\end{definition}

\begin{example} \label{ex:hilbert} Let $t \in T_n$ act on $\CC[z_1, \ldots, z_n]$ as the algebra morphism sending $z_i$ to $t_i z_i$. Consider first the case $n = 1$. The nonzero weight spaces in $\CC[z_1]$ are $\CC[z_1]_{(i)}$ for $i \geq 0$, each of which is one-dimensional, so
the Hilbert series is defined and equal to 
\[\Hs(\CC[z_1]) = \sum_{i=0}^{\infty} a_1^i = 1/(1-a_1).\]
If $V$ and $W$ are representations of $T_m$ and $T_n$, then the Hilbert series of $V \otimes_{\CC} W$ as a $T_m \times T_n$-module is $\Hs(V, a_1, \ldots, a_m)\Hs(W, a_{m+1}, \ldots, a_{m+n})$. In particular, 
\[\Hs(\CC[z_1, \ldots, z_n]) = \prod_{i=1}^n 1/(1-a_i).\]
\end{example}

Let $\CC^{\lambda}$ be the one-dimensional representation of $T$ on which $t \in T$ acts as multiplication by $\lambda(t)$. The \emph{weights} of $V$ are the elements of the unique multiset $\{\lambda_1, \ldots, \lambda_d\}$ such that $V \cong \bigoplus_i \CC^{\lambda_i}$ as a $T$-module. Let $I(Z)$ be the ideal of $Z$ in the coordinate ring $\CC[V] := \Sym(V^*)$. The decomposition of $V$ into one-dimensional weight spaces determines (up to scalars) an isomorphism $\CC[V] \cong \CC[z_1, \ldots, z_d]$ with $z_i \in V_{\lambda_i}$. The $T$-action on $V$ defines a $T$-action on $\CC[V]$, and since $Z$ is $T$-invariant, so is the ideal $I(Z)$.

\begin{theorem}[{\cite[\S6.6]{chriss-ginzburg}}]\label{cone-thm} Suppose that $0$ is not a nontrivial nonnegative linear combination of the weights of $V$. Let $Z \subseteq V$ be a $T$-invariant subscheme. Then 
the quotient $\Hs(\CC[V]/I(Z)) / \Hs(\CC[V])$ is a well-defined polynomial, which corresponds to the class $[Z] \in K_T(V)$ under the isomorphism $K_T(V) \cong R(T) \cong \ZZ[a_1, \ldots, a_n]$.
\end{theorem}
The denominator $\Hs(\CC[V])$ is easily computed: as in Example~\ref{ex:hilbert}, it is the product $\prod_{i=1}^d 1/(1-a^{\lambda_i})$ where $\lambda_1, \ldots, \lambda_d$ are the weights of $V$.

\begin{example} \label{ex:K-poly}
    Take $V = \MO_n$ with $T$-action $t: A \mapsto tAt$ as above. If $e_{ij}$ is the matrix with $1$ in entry $(i,j)$ and $0$ in all other entries, then
    \begin{equation*}
    \MO_n = \bigoplus_{1 \leq i \leq j \leq n} \CC\spanning\{e_{ij} + e_{ji}\}
    \end{equation*}
    decomposes $\MO_n$ into one-dimensional weight spaces. Therefore the monomials $a^{\mu}$ as $\mu$ varies over all weights are $a_i a_j$ for $1 \leq i \leq j \leq n$.
    
    Let $Z$ be the variety of matrices $A\in \MO_n$ with $A_{11} = A_{21} = A_{12} = 0$. Let $z_{ij} : \MO_n \to \CC$ be the map $A \mapsto A_{ij}$, so that
    $\CC[\MO_n] = \CC[z_{ij} : 1 \leq i \leq j \leq n].$ Then $I(Z) = (z_{11}, z_{12})$, so
    \begin{equation*}
        \CC[V]/I(Z) \cong \CC[z_{ij} : 1 \leq i \leq j \leq n, (i,j) \neq (1,1), (1,2) ]
    \end{equation*}
    and hence
    \begin{equation*}
        [Z] = \frac{\Hs(\CC[V]/I(Z))}{\Hs(\CC[V])} = \frac{\displaystyle \prod_{\substack{1\leq i \leq j \leq n \\ (i,j) \neq (1,1),(1,2)}} 1/(1 - a_i a_j)}{\displaystyle \prod_{1\leq i \leq j \leq n} 1/(1 - a_i a_j)} = (1-a_1^2)(1-a_1 a_2).
    \end{equation*}
Note that $Z$ is the symmetric matrix Schubert variety $\iX_z$ for $z = 321 \in I_n$. The preceding calculation shows that $\iG_{321} = (2x_1 + \beta x_1^2)(x_1 + x_2 + \beta x_1x_2)$.
\end{example}

For any polynomials $x$ and $y$, let
\be\label{oplus-def}
x\oplus y := x + y + \beta xy \qquand x\ominus y := \tfrac{x-y}{1+\beta y}
\ee
We say that $z \in \I_n$ is \emph{$\O$-dominant}
if it holds that
\[\DO(z) = \{   (i+j-1,j) \in \PP \times [k] : 1 \leq i \leq \mu_j\}\]
for a strict partition $\mu = (\mu_1 > \mu_2 > \dots > \mu_k > 0)$.
Similarly, we define an involution $z$ to be \emph{$\Sp$-dominant} if $z  \in \ISp_n$ and 
\[\DSp(z) = \{   (i+j,j) \in \PP \times [k] : 1 \leq i \leq \mu_j\}\]
for a strict partition $\mu = (\mu_1 > \mu_2 > \dots > \mu_k > 0)$.
One can show that an involution is $\O$-dominant if and only if 
it is \emph{dominant} in the classical sense of being a $132$-avoiding permutation \cite[Proposition 3.25]{HMP1}.

\begin{theorem} \label{dom-thm}
Let $\K \in \{\O,\Sp\}$
and suppose $z \in \I_n$ is $\K$-dominant. Then \[\KG_z = \prod_{(i,j) \in \DK(z)} x_i\oplus x_j.\] 
\end{theorem}

\begin{proof} 
Assume that $z \in \ISp_n$ if $\K = \Sp$.
It then follows from Proposition~\ref{ess-rpo2} and the fact that $(i,z(i)) \notin \DK(z)$ for all $i \in [n]$
that
$ \KMX_z$ is just the set of matrices $A \in \MK_n$ with $A_{ij} = 0$ for all $(i,j) \in \DK(z)$. Thus, $I(\KMX_z) = \langle z_{ij} : (i,j) \in \DK(z)\rangle$ so $\codim(\KMX_z) = |\DK(z)|$. Exactly as in Example~\ref{ex:K-poly}, this implies that
$    [\KMX_z] = \prod_{(i,j) \in \DK(z)} (1 - a_i a_j) \in K_T(\MK_n)
$, which becomes 
$ \prod_{(i,j) \in \DK(z)} x_i\oplus x_j$ 
on making the transformations in
Definition~\ref{main-def}.
%
%If $\K=\Sp$ and $z \in I_n$ is $\O$-dominant, then the same argument shows that 
%$ \fMX_z$ is the set of matrices $A \in \MSp_n$ with $A_{ij} = 0$ for all $(i,j) \in \DO(z)$.
%In this case, since all elements of $\MSp_n$ are skew-symmetric, the conditions $A_{ii}=0$
%determined by the diagonal positions $(i,i) \in \DO(z) \setminus \DSp(z)$ are redundant,
%so we have  $I(\fMX_z) = \langle z_{ij} : (i,j) \in \DSp(z)\rangle$ and we can proceed as before.
\end{proof}

\begin{remark*}
We believe that when $\K = \O$, Theorem \ref{dom-thm} holds if and only if $z$ is $\O$-dominant, and that when $\K = \Sp$ and $z \in \ISp_n$, Theorem~\ref{dom-thm} holds if and only if $z$ is $\Sp$-dominant. However, Theorem~\ref{dom-thm} can also hold for $\Gfpf_z$ with $z \in \I_n \setminus \ISp_n$. We do not know a nice description of the set of such involutions $z$, but one can show that it includes all the $\O$-dominant involutions.
\end{remark*}

As a special case, we recover two formulas of Wyser and Yong.

\begin{corollary}[Wyser and Yong \cite{WyserYong}]
For any positive integer $n$ it holds that 
\[\iG_{n\cdots 321} = \prod_{1 \leq i \leq j \leq n-i} x_i\oplus x_j
\qquand \Gfpf_{n\cdots 321} = \prod_{1 \leq i < j \leq n-i} x_i\oplus x_j.\]
\end{corollary}

\begin{proof}
This follows by calculating $\DO(n\cdots 321)$ and $\DSp(n\cdots 321)$.
\end{proof}

\subsection{Symplectic Grothendieck polynomials}\label{sp-case-sect}

Throughout this section, we assume that $n \in 2\PP$ is even.
Here, we investigate some properties of the polynomials
$\Gfpf_z$ that are particular to the symplectic case.

Results of Wyser and Yong \cite{WyserYong}
show that 
the family $\{ \Gfpf_z \}_{z \in \ISp_n}$ 
 can be completely characterized in terms of divided difference operators:

\begin{theorem}[{Wyser and Yong \cite{WyserYong}}]
\label{sp-thm1}
The symplectic Grothendieck polynomials $\{ \Gfpf_z \}_{z \in \ISp_\infty}$
are the unique family in $\ZZ[\beta][x_1,x_2,\dots]$
with
\[\Gfpf_{n\cdots 321} = \prod_{1 \leq i < j \leq n - i } (x_i+ x_j+\beta x_ix_j)
\]
and $\varpi_i \Gfpf_z = \Gfpf_{s_iz s_i}$ for all $i \in [n-1]$ such that $i+1 \neq z(i) > z(i+1) \neq i$. 
\end{theorem}

The derivation of Theorem~\ref{sp-thm1} from \cite{WyserYong} requires some explanation.
The \emph{fixed-point-free involution length} of $z \in \ISp_n$ is
\be
\label{ellfpf-eq}
 \ellfpf(z) :=
| \{ (i,j) \in [n]\times [n] :z(i) > z(j) < i < j\}|.
\ee
One has $\ellfpf(z\times 21) = \ellfpf(z) = |\DSp(z)| $,
and the only element $z \in \ISp_n$ with
$\ellfpf(z)=0$ is $\wfpf := s_1s_3s_5\cdots s_{n-1}=(1,2)(3,4)\cdots(n-1,n)$.
It holds that
\be\label{ellfpf-eq2}
\ellfpf(s_izs_i)
=
\begin{cases} 
\ellfpf(z) + 1&\text{if }z(i) < z(i+1) \\
\ellfpf(z) & \text{if }i+1 = z(i) > z(i+1)=i \\
\ellfpf(z)-1 &\text{if }i+1 \neq z(i) > z(i+1) \neq i.
\end{cases}
\ee
By induction, one has 
$%\be\label{ellfpf-eq}
\ellfpf(z) = \min\left \{ \ell(w) : w \in S_n\text{ and } w^{-1} \cdot \wfpf \cdot w  =z\right\}.
$
%where $\wfpf := s_1s_3s_5\cdots s_{n-1} \in \ISp_n$.

\begin{proof}[Proof of Theorem~\ref{sp-thm1}]
Let $a_i = 1-x_i$ and $D_i = \partial^{(-1)}_i$ for $i \in \PP$. 
Wyser and Yong \cite[Theorem 4]{WyserYong}
prove that there exists a unique family of polynomials 
$\{ \Upsilon^{\Sp}_z \}_{z \in \ISp_n} \subseteq \ZZ[x_1,x_2,\dots,x_n]$
with $\Upsilon^{\Sp}_{n\cdots 321}= \prod_{1\leq i < j \leq n - i} (1-a_i a_j)$ 
and $D_i \Upsilon^{\Sp}_z = \Upsilon^{\Sp}_{s_i zs_i}$ for all $i \in [n-1]$
with $i+1 \neq z(i) > z(i+1) \neq i$. 
(In \cite{WyserYong}, the variable $a_i$ is written as $x_i$.)
It is straightforward to check that the elements
\be\label{gfpf-transform}
\Gfpf_z := \beta^{-\ellfpf(z)} \Upsilon^{\Sp}_z(\beta x_1,\beta x_2,\dots,\beta x_n)\ee
belong to $\ZZ[\beta][x_1,x_2,\dots,x_n]$ and make up the unique family with the properties described in 
 Theorem~\ref{sp-thm1}.

It remains to show that these polynomials
are the same as the ones in Definition~\ref{main-def}.
In view of Theorem~\ref{thm:intro-fX-polynomials} and Proposition~\ref{k-ck-prop},
it suffices to verify that 
$\{\Upsilon^{\Sp}_z\}_{z \in \ISp_n}$ 
represent the classes 
of the structure sheaves of $\{ \fX_z\}_{z \in \ISp_n}$ in ordinary $K$-theory
and that $\Upsilon^\Sp_{z} = \Upsilon^\Sp_{z\times 21}$.
This is \cite[Theorems 3 and 4]{WyserYong}.
\end{proof}

We can describe the action of any $\varpi_i$ on $\Gfpf_z$. 

\begin{proposition}\label{fpf-divided-prop}
Let $z \in \ISp_n$ and $i \in [n-1]$. Then 
\[ \varpi _i \Gfpf_z = \begin{cases} \Gfpf_{s_izs_i} &\text{if } i+1\neq z(i) > z(i+1) \neq i \\ 
-\beta\Gfpf_z&\text{otherwise}.
\end{cases}
\]
\end{proposition}

\begin{proof}
Let $y = s_izs_i$.
We have $\varpi _i \Gfpf_z =\Gfpf_y$ if $i+1\neq z(i) > z(i+1) \neq i$
by Theorem~\ref{sp-thm1}.
If $z(i) < z(i+1)$ then
$\Gfpf_z = \varpi_i \Gfpf_y$ so $\varpi_i \Gfpf_z =-\beta\Gfpf_z$ since $\varpi_i\varpi_i=-\beta\varpi_i$.

Now suppose $i+1=z(i)> z(i+1) =i$. To show that $\varpi_i \Gfpf_z = -\beta \Gfpf_z$,
it suffices by \eqref{gfpf-transform}
to check that $s_i \Upsilon^{\Sp}_z = \Upsilon^{\Sp}_z$. 
To show this, we resort to a geometric argument.

The action of $S_n$ on $\ZZ[x_1,x_2,\dots,x_n]$
descends to an action on $K(\Fl_n) \cong \ZZ[x_1,x_2,\dots,x_n]/\ILambda_n$.
Since  $\Upsilon^{\Sp}_z = [\fX_z] \in K(\Fl_n)$ \cite[Theorem 3]{WyserYong},
it follows by Lemma~\ref{ilam-lem}
that we can just show that $s_i[\fX_z] = [\fX_z] \in K(\Fl_n)$.

Let $q : T\backslash \GL_n \to B\backslash \GL_n =: \Fl_n$ be the quotient map.  The left action of $S_n$ on $\GL_n$ which permutes rows descends to $T \backslash \GL_n$ and induces an $S_n$-action on $K(T \backslash \GL_n)$. 
As noted in \eqref{ktkt-eq}, the pullback $q^* : K(\Fl_n) \to K(T \backslash \GL_n)$ is an isomorphism; pulling back the $S_n$-action on $K(T \backslash \GL_n)$ gives the action of $S_n$ on $K(\Fl_n)$ described in the previous paragraph
(see \cite[\S6]{PawlowskiRhoades}).

It is enough to show that $q^*[\fX_z] = [q^{-1}(\fX_z)]$ is $s_i$-invariant. We prove this by showing that the variety $q^{-1}(\fX_z)$ itself is $s_i$-invariant.
Recall that $\Sp_n$ is defined as the subgroup of $\GL_n$ preserving the fixed skew-symmetric
nondegenerate bilinear form $\alpha^\K_n : \CC^n \times \CC^n \to \CC$.
        Proposition~\ref{ess-rpo} implies that
        if $g \in \GL_n$ has rows $g_1,g_2,\dots,g_n$, then 
         $Tg \in q^{-1}(\fX_z)$ if and only if
         the matrix $A = [\alpha_n^\Sp(g_p,g_q)]_{p,q \in [n]}$ has $\rank(A_{[i][j]}) \leq \rank(z_{[i][j]})$ for any $(i,j) \in \Ess(\DSp(z))$.
        These rank conditions are invariant under permuting rows $i$ and $i+1$ of $g$ so long as 
        row $i$ of $\Ess(D^{\Sp}(z))$ is empty. The latter holds since if $(i,j) \in D^{\Sp(z)}$ then
                we have $j < z(i) = i+1$ and $j<i < z(j)$, and therefore also $j < z(i+1) = i$ and $j<i+1 < z(j)$, so $(i+1,j) \in D^{\Sp(z)}$. Thus $q^*[\fX_z]$ is $s_i$-invariant,
                and we conclude that $\varpi_i \Gfpf_z = -\beta \Gfpf_z$.    \end{proof}

Any element of $\ZZ[\beta][[x_1,x_2,\dots]]$ whose homogeneous terms are polynomials
(treating $\beta$ as 
a scalar of degree zero) 
can be uniquely expressed as a possibly infinite $\ZZ[\beta]$-linear combination of ordinary Grothendieck polynomials.

Our next main result shows that the symplectic Grothendieck polynomials have a stronger property: 
each $\Gfpf_z$
is actually a finite linear combination of the polynomials $\fkG_w$ with coefficients in $\{1,\beta,\beta^2,\dots\}$. 
In principle, this could also be deduced from general results of Brion \cite[Theorem 4]{Brion2001}.
One advantage to our approach is that it will let us identify the summands appearing in the expansion
of $\Gfpf_z$ in terms of $\fkG_w$ somewhat 
explicitly.

\def\Nil{\cU_n}
\def\cN{\mathcal{N}_n}

Let $\Nil$ denote the free $\ZZ$-module with a basis given by the symbols 
$U_w$ for $w \in S_n$. Set $U_i  :=U_{s_i}$ for $i \in \PP$.
The abelian group $\Nil$
 has a unique ring structure
 with multiplication satisfying
 \[
U_w U_i := \begin{cases} U_{ws_i} &\text{if }w(i) < w(i+1) \\
U_w & \text{if }w(i) > w(i+1)
\end{cases}
\qquad\text{for $w \in S_n$, $i \in [n-1]$.}
\]
%\[U_{v} U_{w} = U_{vw}\text{ if $\ell(vw) = \ell(v) + \ell(w)$}
%\qquand U_i^2 = U_i\text{ for }i \in \PP.\]
This is
the usual Iwahori-Hecke algebra %$\cH_q = \ZZ[q]\spanning\{T_\pi : \pi \in S_\infty\}$
of $S_n$
 with $q=0$; % and $U_\pi = -T_\pi$
see \cite[Chapter 7]{Humphreys}.

Let $\cN$ be the free $\ZZ$-module with basis $\{  N_z: z \in \ISp_n\}$.
Results of Rains and Vazirani (namely, \cite[Theorems 4.6 and 7.1]{RainsVazirani} with $q=0$)
imply that $\cN$ has a unique structure as a right $\Nil$-module with
\[
N_z U_i := \begin{cases} N_{s_i z s_i} &\text{if }z(i) < z(i+1) \\
N_z &\text{if }i+1 \neq z(i) > z(i+1) \neq i\\
0 &\text{if }i+1=z(i) > z(i+1)=i 
\end{cases}
\quad\text{for $z \in \ISp_n$, $i \in [n-1]$.}
\]
It is shown in \cite{Marberg2016} that (the undegenerated form of) $\cN$
has a ``quasi-parabolic Kazhdan-Lusztig basis''; it would be interesting to relate this basis 
to the polynomials $\iS^\fpf_z := \Gfpf_z|_{\beta=0}$ and $\Gfpf_z$, in analogy with results in \cite{BW2001,Roichman}.

For $z \in \ISp_n$,
define 
$\HAfpf(z) = \{ w \in S_n : N_\wfpf U_w = N_z\}$
where we again let $\wfpf = s_1s_3s_5\cdots s_{n-1}$.
This set is nonempty and
$\ellfpf(z) \leq \ell(w)$ for all $w \in \HAfpf(z)$. Define 
$\cAfpf(z) = \left\{ w\in \HAfpf(z) : \ellfpf(z) = \ell(w)\right\}.$
We refer to the elements of $\cAfpf(z)$ and $\HAfpf(z)$ as \emph{atoms} and \emph{Hecke atoms} for $z$, respectively.
The set $\cAfpf(z)$ consists of
the $w \in S_n$ of minimal length 
with $z = w^{-1} \cdot \wfpf \cdot w$.

\begin{theorem}\label{atom-thm}
If $z \in \ISp_n$ then 
$\Gfpf_z = \sum_{w \in \HAfpf(z)} \beta^{\ell(w) -\ellfpf(z)} \fkG_w$.
\end{theorem}

This result makes it clear that the family $\{\Gfpf_z\}_{z \in \ISp_n}$ is linearly independent. 

\begin{proof}
Define $\Sigma_z := \sum_{w \in \HAfpf(z)} \beta^{\ell(w) -\ellfpf(z)} \fkG_w$ for $z \in \ISp_n$.
 We claim that
\be\label{h-claim-eq} \varpi _i \Sigma_z = \begin{cases} \Sigma_{s_izs_i} &\text{if } i+1\neq z(i) > z(i+1) \neq i \\ 
-\beta\Sigma_z&\text{otherwise}
\end{cases}
\ee
for all $z \in \ISp_n$ and $i \in [n-1]$.
To show this, fix $z \in \ISp_n$ and $i \in [n-1]$
and let $y=s_izs_i \in \ISp_n$.
There are three cases to consider.

First assume $i+1\neq z(i) > z(i+1) \neq i$.
If $w \in \HAfpf(z)$ and $w(i) > w(i+1)$,
then $ws_i \in \HAfpf(x)$ for some $x \in \ISp_n$
and $N_z = N_\wfpf U_w = N_\wfpf U_{ws} U_i = N_x U_i$,
so  $x \in \{y,z\}$ and  $ws_i \in \HAfpf(y) \sqcup \HAfpf(z)$.
Alternatively, 
if $v \in \HAfpf(y)$ then 
$N_\wfpf U_v U_i = N_y U_i = N_z$, so 
$U_v U_i \neq U_v$ and $v(i) < v(i+1)$ and $vs_i \in \HAfpf(z)$.
We conclude that
$\left\{ w \in \HAfpf(z) : w(i) > w(i+1)\right\}$ 
is the disjoint union
\[\label{con-eq} 
%\left\{ w \in \HAfpf(z) : w(i) > w(i+1)\right\} = 
\left\{ vs_i : v \in \HAfpf(y)\right\} \sqcup \left\{ us_i :
 u \in \HAfpf(z),\ u(i) < u(i+1)\right\}.\]
Now, since $\varpi_i \fkG_w =-\beta \fkG_w$ if $w(i) < w(i+1)$, we have
 \[
 \varpi_i \Sigma_z = \sum_{\substack{w \in \HAfpf(z) \\ w(i) >w(i+1)}} \beta^{\ell(w)-\ellfpf(z)} \varpi_i \fkG_w
 -
 \sum_{\substack{w \in \HAfpf(z) \\ w(i) <w(i+1)}} \beta^{\ell(w)-\ellfpf(z)+1}  \fkG_w.
 \]
 Since $\ellfpf(z) = \ellfpf(y)+1$, it follows that the first sum on the right is
\[
%\sum_{\substack{w \in \HAfpf(z) \\ w(i) >w(i+1)}} (-1)^{\ell(w)-\ellfpf(z)} \varpi_i \fkG_w=
\sum_{ v \in \HAfpf(y) } \beta^{\ell(v)-\ellfpf(y)}  \fkG_{v}
+
\sum_{\substack{u \in \HAfpf(z) \\ u(i) <u(i+1)}} \beta^{\ell(u)-\ellfpf(z)+1}  \fkG_u.
 \]
 Substituting this into the previous equation gives $\varpi_i \Sigma_z = \Sigma_{y}$.
 
If $z(i) < z(i+1)$ then $i+1 \neq y(i) > y(i+1) \neq i$, so 
 the previous paragraph
implies that $\Sigma_z = \varpi_i \Sigma_y$ and $\varpi_i \Sigma_z =-\beta\Sigma_z$ as $\varpi_i\varpi_i=-\beta\varpi_i$.
 Finally assume that $i+1 = z(i) > z(i+1) =i+1$. If $w \in \HAfpf(z)$ has $w(i) > w(i+1)$,
 then $ws_i \in \HAfpf(x)$ for some $x \in \ISp_n$ and 
 $N_z =N_\wfpf U_w= N_\wfpf U_{ws} U_i = N_x U_i$,
 which implies the contradiction $0 = N_z U_i = N_x U_i^2 = N_x U_i = N_z$.
 Thus every $w \in \HAfpf(z)$ has $w(i) < w(i+1)$, 
so $\varpi_i \Sigma_z = -\beta\Sigma_z$.
Thus \eqref{h-claim-eq} holds.

We argue by contradiction that $\Gfpf_z = \Sigma_z$ for all $z \in \ISp_n$.
Let $\Delta_z := \Gfpf_z - \Sigma_z$ and suppose $z \in \ISp_n$ is of minimal length $\ellfpf(z)$
such that $\Delta_z \neq 0$.
We cannot have $z = s_1s_3s_5\cdots s_{n-1}$ since then $\Gfpf_z = \Sigma_z =  1$.
The set of indices $I := \{ i \in [n-1] :i+1 \neq z(i) > z(i+1) \neq i\}$
is therefore nonempty.
By Proposition~\ref{fpf-divided-prop}, \eqref{h-claim-eq}, and induction,
we have $\varpi_i \Delta_z = 0$ for all $i \in I$ and $\varpi_i \Delta_z = -\beta \Delta_z$ for all $i \notin I$.
This means that for each $i \in [n-1]$, either $\Delta_z$ or  $(1+\beta x_{i+1})\Delta_z$ 
 is symmetric in $x_i$ and $x_{i+1}$.

The homogeneous term of $\Delta_z$ of lowest degree (with $\deg(x_i) :=1$ and $\deg(\beta):=0$) must therefore be symmetric in $x_1,x_2,\dots,x_n$. 
Since $\Delta_z=\Delta_{z\times 21}$, it follows that $\Delta_z$ 
must actually be symmetric in all the $x_i$-variables for $i \in\PP$.
Since $\Delta_z$
 is a polynomial, this can only occur if $\Delta_z$ has a nonzero constant term.
But it is easy to show by induction that both $\Gfpf_z$ and $\Sigma_z$ have no homogeneous terms of degree less than $\ellfpf(z)\geq 1$, so we reach a contradiction.
Hence no such $z$ can exist, so $\Gfpf_z = \Sigma_z$ for all $z \in \ISp_n$.
\end{proof}

%A \emph{Hecke word}  for $w \in S_n$ is a finite sequence of positive integers $i_1i_2\cdots i_l$ with $U_w = U_{i_1}U_{i_2} \cdots U_{i_l}$.
Given $w \in S_n$,
write $\cH(w)$ for the set of finite integer sequences $i_1i_2\cdots i_l$
with $U_w = U_{i_1}U_{i_2} \cdots U_{i_l}$.
Define $\cHfpf(z) = \bigsqcup_{w \in \HAfpf(z)} \cH(w)$ for $z \in \ISp_n$.
We refer to the elements of $\cH(w)$ (respectively, $\cHfpf(z)$) as \emph{(symplectic) Hecke words}.
One has $i_1i_2\cdots i_l \in \cHfpf(z) $ if and only if $N_z = N_{\wfpf} U_{i_1} U_{i_2}\cdots U_{i_l}$.

\begin{corollary}\label{bjs-cor}
Given a subset $S = \{ (a_1,b_1), (a_2,b_2),\dots,(a_l,b_l) \} \subseteq \PP\times \PP$
with $a_1 \leq a_2 \leq \dots \leq a_l$ and $b_k > b_{k+1}$ whenever $a_k = a_{k+1}$,
define 
\[i_k := b_k + (a_k - 1)
\quand \delta(S) := i_1i_2\cdots i_l
\quand x^S :=x_{a_1} x_{a_2}\cdots x_{a_l}.\] If $z \in \ISp_n$ then
$
\ds\Gfpf_z = \sum_{\substack{S \subseteq [n]\times [n]\\ \delta(S) \in \cHfpf(z)}} 
 \beta^{|S| - \ellfpf(z)} x^S$.
\end{corollary}

\begin{proof}
The formula $\fkG_w =  \sum_{S \subseteq [n]\times [n], \delta(S) \in \cH(w)}  
 \beta^{|S| - \ell(w)} x^S$ for $w \in S_n$
is \cite[Theorem 2.3]{FK1994} (cf. \cite[Corollary 5.4]{KnutsonMiller1}),
so this follows from Theorem~\ref{atom-thm}.
 \end{proof}

The summands $S$ appearing in $\fkG_w =  \sum_{S \subseteq [n]\times [n], \delta(S) \in \cH(w)}  
 \beta^{|S| - \ell(w)} x^S$
 are $K$-theoretic versions of what are usually called \emph{RC-graphs} or \emph{pipe dreams}.

\begin{corollary}
If $z \in \ISp_n$ then $\Gfpf_z \in \NN[\beta][x_1,x_2,\dots,x_n]$.
\end{corollary}

\begin{proof}
This holds as $\fkG_w \in \NN[\beta][x_1,\dots,x_n]$ for all $w \in S_n$ \cite[Theorem 2.3]{FK1994}.
\end{proof}

We can describe $\HAfpf(z)$ more concretely.
Fix an involution $z \in \ISp_n$ and
suppose $a_1<a_2<\cdots$ are the integers $a \in [n]$ such that $a < z(a)$,
arranged in increasing order.
Let $b_i = z(a_i)$ for each $i$ and define 
\[ \betaMin(z) = (a_1b_1 a_2b_2 a_3b_3\cdots)^{-1} \in S_n.\]
Write $w_i =w(i)$ for $w \in S_n$ and $i \in [n-1]$.
Let $\approx_\fpf$ be the strongest equivalence relation on $S_\infty$ 
with $v^{-1} \approx_\fpf w^{-1}$ whenever there is 
an even index $i \in 2\NN$ and
integers $a<b<c<d$ such that 
$v_{i+1}v_{i+2}v_{i+3} v_{i+4}$ and $w_{i+1}w_{i+2}w_{i+3}w_{i+4}$
 both belong to 
$\{ adbc, bcad, bdac\}$
and $v_j = w_j$ for all $j \notin \{i+1,i+2,i+3,i+4\}$.

\begin{proposition}[{\cite[Theorem 2.5]{Mar}}]
\label{ha-prop}
If $z \in \ISp_n$ then 
\[\HAfpf(z) = \left\{ w \in S_n : \betaMin(z) \approx_\fpf w\right\}.\]
\end{proposition}

There is a complementary result for $\cAfpf(z)$.
Let $\prec_\fpf$ be the transitive closure of the relation on $S_n$ 
with $v^{-1} \prec_\fpf w^{-1}$ whenever there is 
an even index $i \in 2\NN$ and
integers $a<b<c<d$ such that 
$v_{i+1}v_{i+2}v_{i+3} v_{i+4} = adbc$
and
$w_{i+1}w_{i+2}w_{i+3}w_{i+4}= bcad$
and $v_j = w_j$ for all $j \notin \{i+1,i+2,i+3,i+4\}$.

\begin{proposition}[{\cite[Theorem 6.10]{HMP2}}]
If $z \in \ISp_n$ then 
\[\cAfpf(z) = \left\{ w \in S_n :  \betaMin(z) \preceq_\fpf w\right\}.\]
\end{proposition}

\begin{example}
The elements of $\HAfpf(z)$ for $z = 4321 \in \ISp_4$ are 
\[\betaMin(4321) =1342 =(1423)^{-1} \prec_\fpf  3124 = (2314)^{-1} \approx_\fpf 3142 = (2413)^{-1}.\]
We have $\ellfpf(4321) = \ell(1342) = \ell(3124) = \ell(3142)-1 = 2$ and
\[
\ba
\fkG_{1342} &= x_1 x_2 + x_1 x_3 + x_2 x_3 + 2\beta x_1 x_2 x_3, \\
\fkG_{3124} &=  x_1^2, \\
\fkG_{3142} &= x_1^2 x_3 + x_1^2 x_2 + \beta x_1^2 x_2 x_3.
\ea
\]
Comparing with Example~\ref{sp-ex1} shows that $\Gfpf_{4321} = \fkG_{1342} + \fkG_{3124} + \beta \fkG_{3142} $.
\end{example}

%\section{More on orthogonal orbit closures}
%\label{o-sect}

\subsection{Degeneracy locus formulas}
\label{subsec:degeneracy-loci}

In contrast to the symplectic case,
the polynomials $\iG_z$ do not have an inductive description
in terms of divided difference operators, and it is an open problem to 
find a general formula for $\iG_z$  that improves on Theorem~\ref{cone-thm}.
We will give a partial solution to this problem in Section~\ref{vex-gro-sect}.
As preparation, 
we review some more general formulas from  \cite{Anderson2017,HIMN} in this section.

% \subsection{Chern classes and raising operators}
%\label{subsec:chern}

 Let $X$ be a smooth complex variety.
Each vector bundle $\cV$ over $X$
has \emph{Chern classes} $c_d(\cV) \in \CK^d(X)$ for $d\in \NN$ and a \emph{Chern polynomial} 
\[c(\cV, t) := \sum_{d \geq 0} c_d(\cV)t^d \in \CK(X)[t]\] with the following properties
(see \cite[Appendix A]{Anderson2017}):
\begin{enumerate}[(a)]
    \item It holds that $c_0(\cV) = 1$ and $c_d(\cV) = 0$ for $d > \rank(\cV)$.
    \item We have $c(\cV,t) = c(\cU, t)c(\cW, t)$ if $0 \to \cU \to \cV \to \cW \to 0$ is a short exact sequence of vector bundles over $X$.
    \item If $\rank (\cV) = 1$, then $c_1(\cV^*) = \frac{-c_1(\cV)}{1 + \beta c_1(\cV)}$.
    \item If $f : Y \to X$ is a morphism, then $c(f^* \cV, t) = f^* c(\cV, t) \in \CK(Y)[t]$.
\end{enumerate}
Since $\CK^d(X) = 0$ for $d>\dim(X)$,
property (a)
implies that $c(\cV, t)$ is invertible in $\CK(X)[t]$,
and any vector bundle $\cV$ over $X=\pt$ must have $c(\cV,t)=1$.
 It follows from property (d),
 with the morphism $Y \to X$ replaced by $X \to \pt$,
  that if $\cV$ is a trivial vector bundle over $X$ then $c(\cV, t) = 1$.
 
Although the difference ``$\cV-\cW$'' for two vector bundles $\cV$ and $\cW$ over $X$ is not defined,
 we set 
 \[c(\cV - \cW, t) := c(\cV, t)/c(\cW, t) \in \CK(X)[[t]].\]
We  regard ``$-$'' defined in this way as a formal inverse of ``$\oplus$,'' 
 which makes sense as property (b) implies that $c((\cV \oplus \cU)- (\cW \oplus \cU), t) = c(\cV - \cW,t)$
 for any vector bundle $\cU$ over $X$.

Let $c_0, c_1, c_2,\ldots$ be indeterminates.
The \emph{raising operator} $T$ associated to such a sequence is the linear operator on the space of arbitrary linear combinations of the $c_i$ variables that sends $c_i\mapsto c_{i+1}$ for each $i$,
and
$\sum_{i \in\NN } a_i c_i \mapsto \sum_{ i \in \NN} a_i c_{i+1}$
for arbitrary coefficients $a_i$.
We adopt the following conventions to make it easier to work with complicated expressions involving these operators:
\begin{itemize}
\item
If $f(x)$ is a function with a Laurent expansion $\sum_{m \in \ZZ} a_m x^m$ at $x=0$,
then we take $f(T)$ to mean $\sum_{m \in \ZZ} a_m T^m$. For instance,
\begin{equation*}
    (1-\beta T)^{-1} c_i := \sum_{m \geq 0} \beta^m T^m c_i = \sum_{m \geq 0} \beta^m c_{m+i}.
\end{equation*}

\item We write $T^{-1}$ for the operator sending 
$\sum_{i \in\NN } a_i c_i \mapsto \sum_{i \in\NN } a_{i+1} c_i$, so that
$T^{-1}(c_i)=c_{i-1}$ for $i > 0$ and $T^{-1}(c_0)=0$.
The composition $T^{-1} T$ is the identity operator while $TT^{-1}$ sends $c_i\mapsto c_i$ for $i>0$ and $c_0\mapsto 0$. 

\item Given a finite collection of sequences of indeterminates $c_0^{(i)}, c_1^{(i)}, c_2^{(i)},\ldots$ for $i \in [n]$, 
we write $T^{(i)}$ for the raising operators that act on monomials by 
\[
\ba
{}&T^{(i)} c_{d_1}^{(1)}\cdots c_{d_i}^{(i)} \cdots c_{d_n}^{(n)} = c_{d_1}^{(1)}\cdots c_{d_i+1}^{(i)} \cdots c_{d_n}^{(n)},
\\
&T^{(i)} c_{d_1}^{(1)}\cdots c_{d_{i-1}}^{(i-1)}c_{d_{i+1}}^{(i+1)} \cdots c_{d_n}^{(n)} = 0.
\ea
\]
In other words, $T^{(i)}$ acts as zero on each monomial that does not involve any 
of $c_0^{(i)}, c_1^{(i)}, c_2^{(i)},\ldots$.
On sums of monomials, $T^{(i)}$ acts linearly in the usual way.
For example, $(T^{(1)})^{-1}(c_2^{(1)}c_1^{(2)} + c_0^{(1)}c_3^{(2)} +  c_1^{(2)}c_2^{(3)}) = c_1^{(1)}c_1^{(2)}$.  This does not define the action of $T^{(i)}$ on a monomial divisible by a product $c^{(j)}_d c^{(j)}_{e\phantom{d}}$, but such monomials will never appear in the cases we consider.

\item To declutter our notation, we sometimes write $1/T^{(i)}$ in place of $(T^{(i)})^{-1}$.
\end{itemize}
 Later we will apply the raising operators $T$ to expressions involving $c_i$ which already have some assigned meaning: in such expressions, we treat the $c_i$ as indeterminates, apply the raising operators, and then replace the symbols $c_i$ with their assigned values.

%\subsection{Degeneracy locus formulas}
%\label{subsec:degeneracy-loci}

Continue to let $X$ denote a smooth complex variety.
Fix $n \in \PP$ and 
let $\pi : \cV \to X$ be a vector bundle of even rank $2n$ over $X$. For $x \in X$, write $\cV_x = \pi^{-1}(x)$ for the fiber of $\cV$ over $x$. Assume $\cV$ is equipped with a nondegenerate skew-symmetric bilinear form, meaning that we have fixed a section of the bundle $\Lambda^2 \cV^*$ which is nondegenerate on each fiber of $\cV$. 

A subbundle $\cF \subseteq \cV$ is \emph{isotropic} with respect to this form if $\cF \subseteq \cF^{\perp}$,
where $\cF^\perp$ is the vector bundle whose fiber over $x \in X$ is the orthogonal complement of $\cV_x$
under the associated form.
Assume $\cF^n \subseteq \cdots \subseteq \cF^1 \subseteq \cV$
is an isotropic flag of subbundles, where $\rank(\cF^i) = n-i+1$, 
and let $\cG \subseteq \cV$
be a maximal isotropic subbundle $\cG \subseteq \cV$, necessarily of rank $n$. 

\begin{definition}\label{lg-def}
For a strict partition $\lambda$ with $\lambda_1 \leq n$, with $\cV$, $\cG$, $\cF^\bullet$ as above,
define the
associated \emph{Lagrangian Grassmannian degeneracy locus} 
to be
\begin{equation} \label{eq:degeneracy-locus}
   \Omega^{\LG}_{\lambda}(\cV, \cG, \cF^{\bullet}):=  \{x \in X : \dim( \cG_x \cap \cF^{\lambda_i}_x ) \geq i \text{ for $i \in [\ell(\lambda)]$} \}.
\end{equation}
\end{definition}
Among the components of $\cF^\bullet$, 
only the bundles $\cF^{\lambda_i}$ play a role in the definition of
 $\Omega^{\LG}_{\lambda}(\cV, \cG, \cF^{\bullet})$. Moreover, as we will discuss in 
 Remark~\ref{rem:essential-rank-conditions},
many of the rank conditions $ \dim( \cG_x \cap \cF^{\lambda_i}_x ) = i$ in \eqref{eq:degeneracy-locus} turn out to be superfluous.

Anderson \cite{Anderson2017} and Hudson, Ikeda, Matsumura, and Naruse \cite{HIMN} give explicit formulas for the 
classes $[\Omega_{\lambda}^{\LG}(\cV, \cG, \cF^{\bullet})] \in \CK(X)$ 
in terms of Chern classes. 

\begin{notation}
In the next theorem, we define certain power series $c^{(i)} \in \CK(X)[[t]]$ 
in the variable $t$. Let $c_d^{(i)} \in \CK(X)$ 
be such that $c^{(i)} = \sum_{d\geq 0} c_d^{(i)} t^d$,
and write $T^{(i)}$ 
for the raising operator acting on $c_d^{(i)}$. 
For $i<j$, we also define 
\be\label{Rij-eq}
R^{(i,j)} := \frac{1 - T^{(i)}/T^{(j)}}{1 + T^{(i)}/ T^{(j)} - \beta T^{(i)}}
.\ee
This operator should be expanded in $T^{(i)}$ as
\[
R^{(i,j)}=
 \sum_{k\geq 0} \(\beta T^{(i)}\)^k 
+ \sum_{\substack{k\geq 0\\ l>0}} (-1)^l \( \tbinom{k+l-1}{k} + \tbinom{k+l}{k}\) \(\beta T^{(i)}\)^k \(T^{(i)}/T^{(j)}\)^l.
\]
Finally, denote the \emph{Pfaffian} of a skew-symmetric matrix $A=(A_{ij})_{i,j\in[r]}$ by
\[\pf(A) :=  \sum_{z \in \ISp_n} (-1)^{\ellfpf(z)} \prod_{z(i) < i \in [r]} A_{z(i),i}.\]
One has $\det(A) = \pf(A)^2$.
 If $n$ is odd then $\pf(A) =0$ since $\ISp_n$ is empty.
\end{notation}

\begin{theorem}[{\cite[Theorem 2]{Anderson2017preprint}}]  \label{thm:anderson-pfaffians} 
Suppose $\lambda$ is a strict partition and $\cV$, $\cG$, and $\cF^\bullet$ are given such that $\Omega_{\lambda}^{\LG}(\cV, \cG, \cF^{\bullet})$ is a Lagrangian Grassmannian degeneracy locus in a smooth complex variety $X$ 
with $\codim(\Omega_{\lambda}^{\LG}(\cV, \cG, \cF^{\bullet})) = |\lambda|$.
Let $r$ be the smallest even integer with $\ell(\lambda) \leq r$.
Let $S$ be a subset of $[r]$ containing
\[
\OC(\lambda) := \{ i \in [\ell(\lambda)-1] : \lambda_i > \lambda_{i+1} + 1\} \sqcup \{ \ell(\lambda)\}.
\]
For $i \in [\ell(\lambda)]$, define
$
        c^{(i)} = c(\cV - \cG - \cF^{\lambda_s}, t)
$
 where $s \in S$ is minimal with $i \leq s$,
and let $c^{(r)} = 1$ and $\lambda_r = 0$ if $r = \ell(\lambda)+1$. Then 
    $
    [\Omega_{\lambda}^{\LG}(\cV, \cG, \cF^{\bullet})]
    \in \CK(X)
    $
     is the Pfaffian of the $r \times r$ skew-symmetric matrix whose $(i,j)$ entry for $i < j$ is
\be \label{eq:pfaffian-entries}
       R^{(i,j)} \(1-\beta T^{(i)}\)^{r-i-\lambda_i}\(1-\beta T^{(j)}\)^{r-j-\lambda_j} c_{\lambda_i}^{(i)} c_{\lambda_j}^{(j)}.
\ee
\end{theorem}

The exponents $r-i-\lambda_i$ and $r-j-\lambda_j$ in \eqref{eq:pfaffian-entries}
may be negative.
Since our statement is slightly different from the one in \cite{Anderson2017preprint},
we sketch a proof below.

\begin{remark}
In the preceding theorem and the proof which follows,
we are citing the {\tt arXiv} version \cite{Anderson2017preprint}
of Anderson's paper rather than the published version \cite{Anderson2017}.
At the time of writing, there is an error in the statement of \cite[Theorem 2]{Anderson2017} which has been corrected
in \cite[Theorem 2]{Anderson2017preprint}.
\end{remark}

\begin{proof}
    When $\ell(\lambda)$ is even, this is the special case of \cite[Theorem 2]{Anderson2017preprint}
    with $s=|S|$ and
    $S = \{ k_1 < k_2 < \dots < k_s\}$, with
     $p_i = 1$ and $q_i = \lambda_{k_i}$ for $i \in [s]$, 
     and with
     $E_{p_i}=\cG$ and $F_{q_i}  = \cF^{q_i}$ for $i \in [s]$.
     When $\ell(\lambda)$ is odd, our matrix is different from the matrix in \cite[Theorem 2]{Anderson2017preprint}, but we
     claim that it has the same Pfaffian. To see this, for any $l\in \PP$ define
$    P^{(l)} := \prod_{1 \leq i < j \leq l} R^{(i,j)} \prod_{i=1}^l \(1-\beta T^{(i)}\)^{l-i-\lambda_i} c^{(i)}_{\lambda_i}
$ where we set $c^{(i)} = 1$ and $\lambda_i = 0$ if $i > \ell(\lambda)$. If $l \geq \ell(\lambda)$, then 
 \[ \(1/T^{(l+1)}\)c^{(l+1)}_{\lambda_{l+1}} = \(1/T^{(l+1)}\)c^{(l+1)}_{0}=0\] and therefore 
\[\ba
        P^{(l+1)} &= \prod_{1 \leq i < j \leq l} R^{(i,j)}  \prod_{i=1}^{l}  R^{(i,l+1)} \prod_{i=1}^{l+1} \(1-\beta T^{(i)}\)^{l+1-i-\lambda_i} c^{(i)}_{\lambda_i}\\
        &= \prod_{1 \leq i < j \leq l} R^{(i,j)} \prod_{i=1}^{l} \(1-\beta T^{(i)}\)^{-1} \prod_{i=1}^{l} \(1-\beta T^{(i)}\)^{l+1-i-\lambda_i} c^{(i)}_{\lambda_i}
        = P^{(l)}.
\ea\]
    It is shown in the proof of \cite[Theorem 2]{Anderson2017preprint} that $P^{(\ell(\lambda))} = [\Omega^{\LG}_\lambda(\mathcal{V}, \mathcal{F}, \mathcal{G}^{\bullet})]$, and that if $l$ is even then $P^{(l)}$ is the Pfaffian of the $l \times l$ skew-symmetric matrix with entries \eqref{eq:pfaffian-entries}. In particular, if $\ell(\lambda)$ is odd then $P^{(\ell(\lambda)+1)}$ is the Pfaffian in the statement of the theorem, and $P^{(\ell(\lambda)+1)} = P^{(\ell(\lambda))} = [\Omega^{\LG}_\lambda(\mathcal{V}, \mathcal{F}, \mathcal{G}^{\bullet})]$.
    \end{proof}

\begin{remark} \label{rem:essential-rank-conditions}
The statement of Theorem~\ref{thm:anderson-pfaffians} 
would be simpler if we fixed $S = [r]$. It is useful to allow some
flexibility in our choice of $S$, however, since not all of the rank conditions in \eqref{eq:degeneracy-locus} are necessary: restricting $i$ to be an element of $S \subseteq [r]$ defines the same locus $\Omega_{\lambda}^{\LG}(\cV, \cG, \cF^{\bullet})$ whenever $\OC(\lambda) \subseteq S$.
Theorem~\ref{thm:anderson-pfaffians} 
gives a formula for 
$[\Omega_{\lambda}^{\LG}(\cV, \cG, \cF^{\bullet})]$ that does not involve
the bundles $\cF^{\lambda_i}$ for $i \notin S$  that are irrelevant to the definition of $\Omega_{\lambda}^{\LG}(\cV, \cG, \cF^{\bullet})$.
\end{remark}

 \subsection{Orthogonal Grothendieck polynomials}\label{vex-gro-sect}

Here, we use
Theorem~\ref{thm:anderson-pfaffians}
to derive a formula for
 the polynomials
 $\iG_z$ indexed by involutions $z \in I_n$ that are \emph{vexillary}
 in the sense of being $2143$-avoiding.
These permutations have a useful alternate characterization;
 recall the definitions of $\Ess(D)$ and $\DO(z)$ from 
 Section~\ref{matrix-sect}.

 \begin{lemma}[{\cite[Lemma 4.18]{Pawlowski}}]\label{vex-lem}
An involution $z \in I_n$ is vexillary if and only if the essential set $\Ess(\DO(z))$) is a chain
  under 
   the partial order $\SWNEleq$ on $\ZZ\times \ZZ$ with $(a,b) \SWNEleq (i,j)$
 if and only if $i \leq a$ and $b \leq j$.
  \end{lemma}

 Write ${\CC^n}^*$ for the dual space of $\CC$-linear maps $\CC^n \to\CC$.
 We represent elements of the direct sum $\CC^n \oplus {\CC^n}^*$ as pairs $(v,\omega)$ where $v \in \CC^n$ and $\omega \in {\CC^n}^*$.
 Define $\langle \cdot,\cdot\rangle^-$ to be the skew-symmetric bilinear form
on ${\CC^n} \oplus {\CC^n}^*$ with 
\be \langle (v_1, \omega_1), (v_2, \omega_2)\rangle^- := \omega_1(v_2) - \omega_2(v_1).\ee
Let $\LG_{2n}$ be the Lagrangian Grassmannian with respect to this form, so that
\[
    \LG_{2n} = \{U \in \Gr(n, \CC^n \oplus {\CC^n}^*) : \langle \cdot,\cdot \rangle^-|_{U \times U} \equiv 0\}.
\]
The \emph{graph} of a bilinear form $\alpha : \CC^n \times \CC^n \to \CC$ is
\be
    \Gamma(\alpha) := \{(v, \alpha(v,\cdot )) : v \in \CC^n\} \in \Gr(n, \CC^n \oplus {\CC^n}^*).
\ee
Such a form $\alpha$ is symmetric if and only if $\graph(\alpha) \in \LG_{2n}$.
Recall from Section~\ref{matrix-sect} that $\alpha_n^\O$ is a fixed 
symmetric nondegenerate bilinear form on $\CC^n$, and that we define
$\O_n$ to be the subgroup of $\GL_n$ preserving $\alpha_n^\O$.

As explained in \cite[\S2.2]{WyserYong},
 geometric obstructions prevent us from being able to completely characterize
 the family $\{\iG_z\}_{z \in I_n}$ using 
  divided difference operators as we did for $\{\Gfpf_z\}_{z \in \ISp_n}$ in Theorem~\ref{sp-thm1}.
  We mention in passing one special situation where the operators $\varpi_i$ 
  do act on $\iG_z$ as one would expect.

\begin{proposition} Let $z \in \I_n$ and $i \in [n-1]$ be such that $z(i) > z(i+1)$.
Assume $z\neq s_i z s_i$ are both vexillary. Then $\varpi_i \iG_z = \iG_{s_i zs_i}$. \end{proposition}

    \begin{proof}
The operators $\varpi_1, \varpi_2,\dots,\varpi_{n-1}$ preserve $\ILambda_n[\beta]$ so descend to operators on $\CK(\Fl_n) = \ZZ[\beta][x_1,\dots,x_n]/\ILambda_n[\beta]$.
It is enough to prove  $\varpi_i[\iX_z] = [\iX_{s_izs_i}] \in \CK(\Fl_n)$,
since then Theorems~\ref{welldef-thm} and \ref{thm:poly-stability}
imply that $\varpi_i \iG_z - \iG_{s_i zs_i} \in \ILambda_N[\beta]$ for all $N \geq n$,
which is only possible if $\varpi_i \iG_z = \iG_{s_i zs_i}$
by Lemma~\ref{ilam-lem}.

        As explained in \cite[\S 2.2]{WyserYong}, 
        to prove that 
        $\varpi_i[\iX_z] = [\iX_{s_izs_i}] \in \CK(\Fl_n)$
        it suffices to show that $\iX_z$ and $\iX_{s_i z s_i}$ have rational singularities. 
        Following our earlier convention,
        given an orbit $E = Bg \in \Fl_n$ where $g \in \GL_n$, let $E_i$ be the subspace of $\CC^n$ spanned by the first $i$ rows of $g$.
        For a subspace $V\subseteq \CC^n$, write $V^\perp$ for the subspace of linear maps in ${\CC^n}^*$ that vanish on $V$.
        Define $\widetilde{X}^{\O}_z  $ to be the closure in $ \LG_{2n} \times \Fl_n$ of the set of pairs $(U,E) \in \LG_{2n}\times \Fl_n$ satisfying
        \[
        \dim(U \cap (E_j \oplus E_i^\perp)) = j - \rank(z_{[i][j]}) \text{ for all $(i,j) \in \Ess(\DO(z))$}.
        \]
        Since $z$ is vexillary,  the elements of $\Ess(\DO(z))$ form a chain $(i_1,j_1),\dots,(i_s,j_s)$ in the order $\SWNEleq$ from Lemma~\ref{vex-lem}.
        If $E \in \Fl_n$ then 
            \begin{equation*}
            E_{j_1} \oplus E_{i_1}^{\perp} \subseteq \cdots \subseteq E_{j_s} \oplus E_{i_s}^{\perp}
        \end{equation*}
        is an isotropic flag in $\CC^n \oplus {\CC^n}^*$.
        This makes it clear than 
        the fiber over $E \in \Fl_n$ of the obvious projection $\widetilde{X}^{\O}_z \to \Fl_n$  is isomorphic to a Schubert variety in $\LG_{2n}$. Schubert varieties have rational singularities \cite[\S 8.2.2]{Kumar}, so the same is true of $\widetilde{X}^{\O}_z$ by \cite[Th\'eor\`eme 2]{Elkik}.

        Let $\iota : \Fl_n \hookrightarrow \LG_{2n} \times \Fl_n$ be the inclusion $E \mapsto (\graph(\alpha^{\O}_n), E)$. We claim that $\iX_z$ is the scheme-theoretic fiber $\iota^{-1}(\widetilde{X}^{\O}_z)$.  It follows from Lemma~\ref{lem:degeneracy-loci} that $\iota^{-1}(\widetilde{X}^{\O}_z)$ and $\iX_z$ agree as sets, so it suffices to show that $\iota^{-1}(\widetilde{X}^{\O}_z)$ is reduced. Let $\pi : \widetilde{X}^{\O}_z \to \LG_{2n}$ be projection onto the first factor. Then $\iota$ is an isomorphism $\iota^{-1}(\widetilde{X}^{\O}_z) \to \pi^{-1}(\graph(\alpha^{\O}_n))$, and the fiber $\pi^{-1}(\graph(\alpha^{\O}_n))$ is reduced because $\pi$ is a fiber bundle over an open subset of $\LG_{2n}$ containing $\graph(\alpha^{\O}_n)$ \cite[Lemma 5.3]{Pawlowski}. This establishes the claim, so  $\iX_z$ has rational singularities by \cite[Th\'eor\`eme 3]{Elkik} and the fact that $\widetilde{X}^{\O}_z$ has rational singularities.
The same argument applied to $s_izs_i$ shows that $\iX_{s_izs_i}$ also has rational singularities, so we have $\varpi_i[\iX_z] = [\iX_{s_izs_i}] \in \CK(\Fl_n)$
by the discussion in \cite[\S2.2]{WyserYong}.
  \end{proof}

The \emph{orthogonal code} of $z \in \I_n$ is the sequence $c^\O(z) = (c_1,c_2,\dots,c_n)$
where $c_i$ %= |\{ j \in [n] : z(i) > z(j) \leq i < j\}|$
 is the number of positions
in the $i$th row of $ \DO(z)$.
The \emph{orthogonal shape} $\shO(z)$ of $z \in \I_n$ is the transpose of the partition 
sorting $c^\O(z)$.
These objects are denoted $\hat c(z)$ and $\mu(z)$ in \cite[\S4.3]{HMP4}.
If  $z = n\cdots 321 \in \I_n$,
for example,
then we have
$\shO(z) = (n-1,n-3,n-5,\dots).$

\begin{lemma}[{See \cite[\S5.2]{Pawlowski}}] \label{lem:degeneracy-loci} Suppose $z \in I_n$ is vexillary so that 
\[\Ess(\DO(z)) = \{(i_1, j_1) \SWNEneq (i_2, j_2) \SWNEneq \cdots \SWNEneq (i_s, j_s)\}\]
where $\SWNEleq$ is the order in Lemma~\ref{vex-lem}.
Let $\lambda$, $\cV$, $\cG$, and $\cF^\bullet $ be given as follows:
    \begin{itemize}
    \item[(i)] Define $\lambda = \shO(z)$.
    
        \item[(ii)] Define $\cV$ to be the trivial bundle $\CC^n \oplus {\CC^n}^*$ over $\Fl_n$ equipped with the skew-symmetric form $\langle \cdot,\cdot\rangle^-$.
        
        \item[(iii)] Define $\cG$ to be the trivial bundle $\Gamma(\alpha_n^{\O})$ over $\Fl_n$.
        
        \item[(iv)] Let $\cF^\bullet$  denote the flag $\cE_{j_1} \oplus \cE_{i_1}^{\perp} \subseteq \cdots \subseteq \cE_{j_s} \oplus \cE_{i_s}^{\perp}$,   where $\cE_i$ for $i \in [n]$ is the tautological bundle of $\Fl_n$
        whose fiber over an orbit $Bg \in \Fl_n$ for $g \in \GL_n$ is the subspace of $\CC^n$ spanned by the first $i$
        rows of $g$.
    \end{itemize}
Then, in the notation of Definition~\ref{lg-def} (with the ambient variety given by $\Fl_n$), we have $X^{\O}_z=\Omega^{\LG}_{\lambda}(\cV, \cG, \cF^{\bullet})$.
\end{lemma}

\begin{proof}
It suffices to show that the rank conditions defining $X^{\O}_z$
(given in terms of $z$) are equivalent to the rank conditions defining $\Omega^{\LG}_{\lambda}(\cV, \cG, \cF^{\bullet})$ (given in terms of $\lambda$),
and this follows from \cite[Lemmas 5.5 and 5.6]{Pawlowski}.
\end{proof}

Let $\bar x := 0\ominus x = \frac{-x}{1+\beta x}$
where $\ominus$ is as in \eqref{oplus-def}.
For each $z \in I_n$, let 
\[S(z) := \{q - \rank(z_{[p][q]}) : (p,q) \in \Ess(\DO(z))\}.\]

\begin{lemma}
Assume $z \in I_n$ is vexillary. The following properties hold: 
\begin{enumerate}[(a)]
    \item $|S(z)| = |\Ess(\DO(z))|$.
    \item $S(z)$ contains the set $\OC(\shO(z))$ defined in Theorem~\ref{thm:anderson-pfaffians}.
\end{enumerate}
\end{lemma}

\begin{proof} For part (a), we observe that  the Rothe diagram $D(z)$ is formed by removing from $[n]\times [n]$ all positions
$(i+j,z(i))$ and $(i,z(i)+j)$  for $i \in [n]$ and $j\geq 0$,
and $\DO(z)$ is the subset of positions $(p,q) \in D(z)$ with $p \geq q$.
Suppose $(p_1,q_1),(p_2,q_2)\in \Ess(\DO(z))$ are distinct.
By Lemma~\ref{vex-lem}, we may assume that $p_1 \geq p_2$ and $q_1 \leq q_2$.

It suffices to show that $\rank(z_{[p_2][q_2]}) - \rank(z_{[p_1][q_1]}) < q_2 - q_1$. 
If $q_1=q_2$ then clearly $\rank(z_{[p_2][q_2]}) - \rank(z_{[p_1][q_1]}) \leq 0$
and we cannot have equality since this would imply that
$(i,q_1) = (i,q_2) \in \DO(z)$ for all $p_2 \leq i \leq p_1$, contradicting $(p_2,q_2)\in \Ess(\DO(z))$.
If $q_1 < q_2$ then $\rank(z_{[p_2][q_2]}) - \rank(z_{[p_1][q_1]})$ is bounded above by the number of pairs $(i,z(i))$ with $1\leq i \leq p_2$ and $q_1 < z(i) \leq q_2$,
which is at most $q_2-q_1-1$ since no such pair has $z(i) = q_2$ as $(p_2,q_2) \in \DO(z)$.

For part (b), write $\DO(z) = \{(i_1, j_1) \prec \cdots \prec (i_s, j_s)\}$ where $\prec$ is the order defined in Lemma~\ref{vex-lem}. Also let $S(z) = \{k_1 < \cdots < k_s\}$. Then $k_s = \ell(\shO(z))$ and for any $k \leq \ell(\shO(z))$, we have $\shO(z)_k = i_p - j_p + 1 + k_p - k$ where $p$ is such that $k_{p-1} < k \leq k_p$ \cite[Lemma 4.23]{Pawlowski}. This implies that if $k \notin S(z)$, then $\shO(z)_k = \shO(z)_{k+1}+1$.
\end{proof}

\begin{theorem} \label{thm:GO-pfaffian} 
Suppose $z \in I_n$ is a vexillary involution with shape $\lambda = \shO(z)$.
Let $r$ be the smallest even integer with $\ell(\lambda) \leq r$.
For each $i \in [\ell(\lambda)]$, let
\be\label{ci-eq}
c^{(i)}  =   \sum_{d\geq 0} c_d^{(i)} t^d  
 := \prod_{m=1}^{p} (1 + x_m t) \prod_{m=1}^{q} \left(1 + \bar{x}_m t\right)^{-1}
 \ee
 where 
 $(p,q) \in \Ess(\DO(z))$ is such that $q - \rank(z_{[p][q]})  = \min\{ s \in S(z): i \leq s\}$. If $r = \ell(\lambda)+1$ then also let $c^{(r)} = 1$.
The polynomial $\iG_z$ is then the Pfaffian of the $r \times r$  skew-symmetric matrix whose $(i,j)$ entry for $i < j$ is
\be
\label{fkm-eq}
       R^{(i,j)}
        \(1-\beta T^{(i)}\)^{r-i-\lambda_i}\(1-\beta T^{(j)}\)^{r-j-\lambda_j} c_{\lambda_i}^{(i)} c_{\lambda_j}^{(j)}
        \ee
where $T^{(i)}$ 
is the raising operator acting on $c_d^{(i)}$
and $R^{(i,j)}$ is defined by \eqref{Rij-eq}.
\end{theorem}

\begin{proof}
    Since $\codim(X^{\O}_z) = |\shO(z)|$ (see \cite[Theorem 4.6]{RichSpring}),
 Theorem~\ref{thm:anderson-pfaffians} and Lemma~\ref{lem:degeneracy-loci} imply that $[\iX_z] \in \CK(\Fl_n)$ is the Pfaffian of the $r\times r$ skew-symmetric matrix $\fk M$ with entries \eqref{eq:pfaffian-entries} where $\lambda = \shO(z)$,
 $S = S(z)$, and
$
        c^{(i)} := c(\cV - \cG - (\cE_q \oplus \cE_p^{\perp}), t),
$
where $p$ and $q$ are such that
$(p,q) \in \Ess(\DO(z))$ and $q - \rank(z_{[p][q]})  = \min\{ s \in S(z): i \leq s\}$. Using the triviality of $\cV$ and $\cG$, the canonical isomorphism $E_i^{\perp} \cong (\CC^n/E_i)^*$, and the basic properties of Chern classes presented in Section~\ref{subsec:degeneracy-loci}, we deduce that 
    \begin{align*}
        c^{(i)}  = \frac{1}{c(\cG, t) c(\cE_q,t) c(\cE_p^{\perp},t)} = \frac{c(\cE_p^*, t)}{c(\cE_q, t)} = \frac{\prod_{m=1}^p (1 + x_m t)}{\prod_{m=1}^q (1 + \bar{x}_m t)}.
    \end{align*}
    Thus $c^{(i)}$ is as in \eqref{ci-eq}, so 
    $\fk M$ is the skew-symmetric matrix with entries \eqref{fkm-eq},
    and we have $ \pf(\fk M) = [\iX_z] \in \CK(\Fl_n)$.
 
Let $\ILambda'_n$ denote the ideal in $\ZZ[[x_1,\dots,x_n]]$
generated by the symmetric formal power series,
so that $\ILambda_n = \ILambda'_n \cap \ZZ[x_1,\dots,x_n]$.
The entries of $\fk M$, and therefore also $\pf(\fk M)$,
belong to the ring of formal power series $\ZZ[\beta][[x_1,\dots,x_n]]$,
and the assertion  $ \pf(\fk M) = [\iX_z] \in \CK(\Fl_n)$
means that $\pf(\fk M) \in \iG_z +  \ILambda'_n[\beta]$.
We claim that in fact $\pf(\fk M) = \iG_z$ as polynomials.

Theorem~\ref{thm:poly-stability} and Lemma~\ref{ilam-lem} imply that $\iG_z$ is the unique polynomial with $\iG_z + \ILambda_N[\beta] = [\iX_z] \in \CK(\Fl_N)$ for all $N\geq n$. In fact, 
it follows that $\iG_z$ is unique among 
formal power series in $\ZZ[\beta][[x_1,x_2,\dots]]$ 
that are polynomials in each fixed degree such that
$\iG_z + \ILambda_N'[\beta]$ coincides with the image of $[\iX_z]$ 
under the inclusion $\ZZ[\beta][x_1,\dots,x_N]/\ILambda_N[\beta] \hookrightarrow \ZZ[\beta][[x_1,\dots,x_N]]/\ILambda'_N[\beta]$ for all $N\geq n$.
But $\pf(\fk M)$ also has this property, since 
$\fk M$ does not change if we replace $z$ by $z\times 1$. 
We must therefore have $\pf(\fk M) = \iG_z$.
\end{proof}

\begin{example} \label{ex:GO-pfaffian} 
Let $z = 21\in I_2$ 
 so  $\lambda = (1)$, $r = 2$, and $\Ess(\DO(z)) = \{(1,1)\}$. Then
$
        c^{(1)} = \frac{1 + x_1 t}{1 + \bar{x}_1 t} 
        $
        and
        $ c^{(2)} = 1$
and $\iG_{21} = \pf \left[\begin{smallmatrix} 0 & f \\ -f & 0 \end{smallmatrix}\right] = f
:= R^{(1,2)}c_1^{(1)} c_0^{(2)}.
$
    Since $1/T^{(2)}$ annihilates $c_1^{(1)}c_0^{(2)}$ and since $c_0^{(2)} = 1$, we have 
    \begin{equation*}
        f = \tfrac{1}{1 - \beta T^{(1)}} c_1^{(1)} = \sum_{m \geq 0} \beta^m c_{m+1}^{(1)} = \beta^{-1}\( \tfrac{1 + x_1 \beta}{1 + \bar{x}_1 \beta} - 1\) = 2x_1 + \beta x_1^2.
    \end{equation*}
    This gives $\iG_{21} =  2x_1 + \beta x_1^2$ which agrees with Example~\ref{o-ex}.
\end{example}

Example~\ref{ex:GO-pfaffian} required a little algebra to simplify the infinite sums resulting from Theorem~\ref{thm:GO-pfaffian} to polynomials.  We now describe a change of variable which handles these simplifications in general.

We have been working with certain expressions $c_m^{(i)}$ that we often view as formal indeterminates.
Let $D_1,D_2,D_3\dots$ be another sequence of commuting indeterminates, and 
if $f$ is a linear combination of monomials $c_{m_1}^{(1)} \cdots c_{m_\ell}^{(\ell)}$, then 
define $\Phi(f)$ to be the formal sum obtained by replacing each $c_{m_i}^{(i)}$ by $D_i^{m_i}/m_i!$.
 Then $\Phi((1/T^{(i)})f) = \frac{\partial}{\partial D_i} \Phi(f)$ and
\be\label{int-eq}
    \Phi(T^{(i)}f) = \int_0^{D_i} f(D_1, \ldots, D_{i-1}, u, D_{i+1}, \ldots, D_l)\,du.
\ee
For example, we have
\begin{equation*}
    \Phi\(\(1-\beta T^{(1)}\)^{-1}c_1^{(1)}\) = \Phi\left(\sum_{m\geq 0}\beta^m c_{m+1}^{(1)}\right) = \sum_{m \geq 0} \frac{\beta^m D_1^{m+1}}{(m+1)!} =  \frac{e^{\beta D_1} - 1}{\beta}.
\end{equation*}
For integers $r \in \PP$ and $a \in \ZZ$, define 
    \begin{equation*}
        F_{r,a}(D) := \tfrac{1}{(r-1)!}\left( \tfrac{\partial}{\partial D} \right)^{r-a-1} \( D^{r-1} e^{\beta D} \)
    \end{equation*}
    where $(\frac{\partial}{\partial D})^{-1} f(D) := \int_0^{D} f(u)\,du$ and $(\frac{\partial}{\partial D})^{-m} := ((\frac{\partial}{\partial D})^{-1})^{m}$ for $m > 0$. We also set $F_{0,a}(D) = D^a/a!$.

\begin{proposition} \label{prop:phi-pfaffian} For any integers $r,s \in \PP$ and $a,b \in \ZZ$, the expression 
    \begin{equation} \label{eq:phi-pfaffian}
        \Phi\(R^{(1,2)} \(1-\beta T^{(1)}\)^{-r}\(1-\beta T^{(2)}\)^{-s} c_a^{(1)}c_b^{(2)} \)
    \end{equation}
    is equal to
    \begin{equation*}
        e^{\beta D_1} \int_0^{D_1} e^{-\beta u} \Bigl(F_{r,a-1}(u)F_{s,b}(u+D_2-D_1) - F_{r,a}(u)F_{s,b-1}(u+D_2-D_1) \Bigr) du.
    \end{equation*}
\end{proposition}

\begin{proof}
Using the fact that $(1-\beta x)^{-r} = \sum_{k=0}^{\infty} {r+k-1 \choose r-1} \beta^k x^k$ 
it is routine to verify $\Phi((1-\beta T^{(i)})^{-r} c_a) = F_{r,a}(D)$. Set 
$M := \(1-\beta T^{(1)}\)^{-r}\(1-\beta T^{(2)}\)^{-s} c_a^{(1)}c_b^{(2)}$ so that $\Phi(M) = F_{r,a}(D_1) F_{s,b}(D_2)$, and define
    \begin{align*}
        \Theta(D_1, D_2) = \tfrac{\partial \Phi(M)}{\partial D_1} - \tfrac{\partial \Phi(M)}{\partial D_2} = F_{r,a-1}(D_1)F_{s,b}(D_2) - F_{r,a}(D_1)F_{s,b-1}(D_2).
    \end{align*}
Now let $G(D_1, D_2)$ be the expression in \eqref{eq:phi-pfaffian}. We have
\[
\ba
        \left(\tfrac{\partial}{\partial D_1} + \tfrac{\partial}{\partial D_2} - \beta\right)G(D_1, D_2) &= \Phi\((1/T^{(1)} +1/T^{(2)} - \beta)R^{(1,2)} M\)   \\
        &= \Phi\((1/T^{(1)}-1/T^{(2)})M\) \\&= \Theta(D_1, D_2).
\ea
\]
     The rational function in $T^{(1)}$ and $ T^{(2)}$ appearing inside $\Phi$ in \eqref{eq:phi-pfaffian} only involves nonnegative powers of $T^{(1)}$ when expanded as a Laurent series in $T^{(1)}$, so we have $G(0,D_2) = 0$. Thus, if we define $\tilde{G}(u) := G(u, u+D_2-D_1)$, then $\tilde{G}(0) = 0$. By the multivariate chain rule and
     our expression for $ \(\tfrac{\partial}{\partial D_1} + \tfrac{\partial}{\partial D_2} - \beta\)G(D_1, D_2)$ derived above,
we deduce that
     \begin{align*}
        \tfrac{\partial }{\partial u}\tilde{G} - \beta \tilde{G} &= \tfrac{\partial G}{\partial D_1}(u, u+D_2-D_1) + \tfrac{\partial G}{\partial D_2}(u, u+D_2-D_1) - \beta \tilde{G}\\
        &= \Theta(u, u+D_2-D_1).
     \end{align*}
    Therefore $\tilde{G}(u)$ is the unique solution to the initial value problem $\tfrac{\partial }{\partial u} \tilde{G}-\beta \tilde{G} = \Theta(u, u+D_2-D_1)$ and $\tilde{G}(0) = 0$, which one checks to be 
     \begin{equation*}
        e^{\beta u} \int_0^{u} e^{-\beta v} \Theta(v, v+D_2-D_1) dv.
     \end{equation*}
         As $G(D_1,D_2) = \tilde{G}(D_1)$, the result follows.
    \end{proof}

The next proposition gives an algorithm for computing the inverse map $\Phi^{-1}$.

\begin{proposition} \label{prop:phi-inverse}
Suppose $G$ is a formal infinite linear combination of monomials in the $D_i$ with coefficients in $\QQ[\beta]$. The
following properties then hold:
\begin{enumerate}[(a)]
    \item Interpreting $D_i$ as  $\frac{\partial}{\partial t_i}\big|_{t_i=0}$, we have $\Phi^{-1}(G) = G\(c^{(1)}(t_1)\cdots c^{(l)}(t_l)\)$.
    
    \item Interpreting $D_i$ as  $\frac{\partial}{\partial t}\big|_{t=0}$, we have  
    $\Phi^{-1}(D_i^m e^{\beta D_i}) = \frac{\partial^m}{\partial t^m} c^{(i)}\big|_{t= \beta}$. 

\end{enumerate}
\end{proposition}

\begin{proof}
    For part (a), observe that $\Phi^{-1}(D_i^m/m!) = c_m^{(i)} = \frac{1}{m!}\frac{d^m}{dt^m} c^{(i)} \big|_{t=0}$. Part (b)
    holds since we have
$
    D_i^m e^{\beta D_i} c^{(i)} = \sum_{d\geq 0} \tfrac{\beta^d}{d!}  \tfrac{\partial^{m+d}}{\partial t^{m+d}}c^{(i)}\big|_{t=0} =\tfrac{\partial^m}{\partial t^m} c^{(i)}\big|_{t=\beta}.
$
\end{proof}

\begin{example} \label{ex:integrals}
    Let us compute $\iG_z$ for $z = 3412 = (1,3)(2,4)$. We have
\[
 \DO(z) = 
      \boxed{  \begin{array}{cccc}
            \circ & \cdot & \times & \cdot\\
            \circ & \circ & \cdot & \times\\
            \times & \cdot & \cdot & \cdot\\
            \cdot & \times & \cdot & \cdot
        \end{array}
        }
\]
so $\Ess(\DO(z)) = \{(2,2)\}$ and $\shO(z) = (2,1)$. Theorem~\ref{thm:GO-pfaffian}
 implies that
\[
\iG_{3412}= R^{(1,2)} \(1-\beta T^{(1)}\)^{-1} \(1-\beta T^{(2)}\)^{-1} c_2^{(1)} c_1^{(2)}
\]
    where 
$
        c^{(1)} = c^{(2)} = \frac{1+x_1 t}{1 + \bar{x}_1 t} \frac{1+x_2 t}{1 + \bar{x}_2 t}.
$
Following Proposition~\ref{prop:phi-pfaffian}, we have
   \[
        \Theta(D_1, D_2) = F_{1,1}(D_1)F_{1,1}(D_2) - F_{1,2}(D_1)F_{1,0}(D_2).\]
        The terms $F_{r,a}(D_i)$ on the right are
        \[
        \ba
        F_{1,1}(D_1)&= \int_0^{D_1} e^{\beta u}\,du = \beta^{-1} e^{\beta D_1} - \beta^{-1},
        \\
        F_{1,1}(D_2) &=\int_0^{D_2} e^{\beta u}\,du=\beta^{-1} e^{\beta D_2} - \beta^{-1},
        \\
        F_{1,2}(D_1)&=\int_0^{D_1} \int_0^{v} e^{\beta u}\,du\,dv = \beta^{-2} e^{\beta D_1}  - \beta^{-2} - \beta^{-1} D_1,\\
        F_{1,0}(D_2)&= e^{\beta D_2},
        \ea\]
so 
$ \Theta(D_1, D_2)=  \beta^{-1} D_1 e^{\beta D_2} - \beta^{-2} e^{\beta D_1} + \beta^{-2}$
%  \[
%   \ba
%        \Theta(D_1, D_2) &= F_{1,1}(D_1)F_{1,1}(D_2) - F_{1,2}(D_1)F_{1,0}(D_2)  \\
%        &= \int_0^{D_1} e^{\beta u}\,du \int_0^{D_2} e^{\beta u}\,du - \int_0^{D_1} \int_0^{v} e^{\beta u}\,du\,dv \cdot e^{\beta D_2}
%        %
%        %
%        \\& = (\beta^{-1} e^{\beta D_1} - \beta^{-1})(\beta^{-1} e^{\beta D_2} - \beta^{-1}) - (\beta^{-2} e^{\beta D_1} - \beta^{-1} D_1 - \beta^{-2})e^{\beta D_2}\\
%        &= 
%        \beta^{-1} D_1 e^{\beta D_2} - \beta^{-2} e^{\beta D_1} + \beta^{-2}
%\ea
%\]
    and then
\[\ba
        \Phi(\iG_{3412}) &= e^{\beta D_1} \int_0^{D_1} e^{-\beta u} \Theta(u, u+D_2-D_1) du\\
        &= \tfrac{1}{2} \beta^{-1} D_1^2 e^{\beta D_2} - \beta^{-2} D_1 e^{\beta D_1} + \beta^{-3} e^{\beta D_1} - \beta^{-3}.
\ea\]
Finally, using Proposition~\ref{prop:phi-inverse}, we compute that
    \begin{align*}
        \iG_{3412} &= \Phi^{-1}\( \tfrac{1}{2} \beta^{-1} D_1^2 e^{\beta D_2} - \beta^{-2} D_1 e^{\beta D_1} + \beta^{-3} e^{\beta D_1} - \beta^{-3}\) 
%        \\ &=\Phi(f)c^{(1)}(t_1)c^{(2)}(t_2) 
        \\&= \tfrac{1}{2}\beta^{-1} \cdot \tfrac{\partial^2}{\partial t^2} c^{(1)} \big|_{t=0} \cdot c^{(2)}\big|_{t=\beta} - \beta^{-2}\cdot  \tfrac{\partial}{\partial t} c^{(1)}  \big|_{t=\beta} + \beta^{-3} \cdot c^{(1)}\big|_{t=\beta}- \beta^{-3}\\
        &= 4x_1 x_2^2 + 4x_1^2 x_2 + 2\beta x_1 x_2^3 + 8\beta x_1^2 x_2^2 + 2\beta x_1^3 x_2 + 3\beta^2 x_1^2 x_2^3
        \\&\quad\ + 3\beta^2 x_1^3 x_2^2 + \beta^3 x_1^3 x_2^3 \\
        & = (x_1 \oplus x_1)(x_1\oplus x_2)(x_2\oplus x_2)
    \end{align*}
which agrees with Theorem~\ref{dom-thm}.
\end{example}

\begin{example}
    Let $z = 4571263= (1,4)(2,5)(3,7)  \in I_7$, so that
    \begin{equation*} \DO(z) = 
    \boxed{
        \begin{array}{ccccccc}
            \circ & \cdot & \cdot & \times & \cdot & \cdot & \cdot\\
            \circ & \circ & \cdot & \cdot & \times & \cdot & \cdot\\
            \circ & \circ & \circ & \cdot & \cdot & \cdot & \times\\
            \times & \cdot & \cdot & \cdot & \cdot & \cdot & \cdot\\
            \cdot & \times & \cdot & \cdot & \cdot & \cdot & \cdot\\
            \cdot & \cdot & \circ & \cdot & \cdot & \times & \cdot\\
            \cdot & \cdot & \times & \cdot & \cdot & \cdot & \cdot
        \end{array}
        }
    \end{equation*}
    and $\Ess(\DO(z)) = \{(6,3), (3,3)\}$ and $\shO(z) = (4,2,1)$. In the notation of Theorem~\ref{thm:GO-pfaffian}, 
    one has $r = 4$, $S = \{1,3\}$, $c^{(4)} = 1$,  
\[
        c^{(1)} = \tfrac{(1+x_1 t)(1+x_2 t)\cdots (1+x_6 t)}{(1 + \bar{x}_1 t)(1 + \bar{x}_2 t)(1 + \bar{x}_3 t)},
\quand
 c^{(2)} = c^{(3)} = \tfrac{(1+x_1 t)(1 + x_2 t)(1+x_3 t)}{(1 + \bar{x}_1 t)(1 + \bar{x}_2 t)(1 + \bar{x}_3 t)}.
\]
    %Let $R^{(i,j)} = \frac{1-T^{(i)} /T^{(j)}}{1 + T^{(i)} /T^{(j)} - \beta T^{(i)}}$. 
    Theorem~\ref{thm:GO-pfaffian} tells us that $\iG_z = \pf(\fk M)$ where $\fk M$ is the $4 \times 4$ skew-symmetric matrix with entries
    \begin{align*}
        \fk M_{12} &:=  R^{(1,2)}\(1-\beta T^{(1)}\)^{-1} \(1-\beta T^{(2)}\)^{0} c_4^{(1)} c_2^{(2)},\\
        \fk M_{13} &:= R^{(1,3)}\(1-\beta T^{(1)}\)^{-1} \(1-\beta T^{(3)}\)^{0} c_4^{(1)} c_1^{(3)},\\
        \fk M_{14} &:= R^{(1,4)} \(1-\beta T^{(1)}\)^{-1} \(1-\beta T^{(4)}\)^{0} c_4^{(1)} c_0^{(4)},\\
        \fk M_{23} &:= R^{(2,3)} \(1-\beta T^{(2)}\)^{0} \(1-\beta T^{(3)}\)^{0} c_2^{(2)} c_1^{(3)},\\
        \fk M_{24} &:= R^{(2,4)} \(1-\beta T^{(2)}\)^{0} \(1-\beta T^{(4)}\)^{0} c_1^{(2)} c_0^{(4)},\\
        \fk M_{34} &:= R^{(3,4)} \(1-\beta T^{(3)}\)^{0} \(1-\beta T^{(4)}\)^{0} c_1^{(3)} c_0^{(4)}.
    \end{align*}
Calculating as in Example~\ref{ex:integrals}, we
find that 
$\iG_{4571263} \in \NN[\beta][x_1, x_2,\dots, x_6]$ is a polynomial with 865 terms
which begins as
\begin{align*}
&      \beta^{11} x_{1}^{5} x_{2}^{5} x_{3}^{5} x_{4} x_{5} x_{6} + \beta^{10} x_{1}^{5} x_{2}^{5} x_{3}^{5} x_{4} x_{5} + \beta^{10} x_{1}^{5} x_{2}^{5} x_{3}^{5} x_{4} x_{6} + \beta^{10} x_{1}^{5} x_{2}^{5} x_{3}^{5} x_{5} x_{6}  \\
        & + 5  \beta^{10} x_{1}^{5} x_{2}^{5} x_{3}^{4} x_{4} x_{5} x_{6} + 5  \beta^{10} x_{1}^{5} x_{2}^{4} x_{3}^{5} x_{4} x_{5} x_{6} + 5  \beta^{10} x_{1}^{4} x_{2}^{5} x_{3}^{5} x_{4} x_{5} x_{6} 
        \\& + ( \text{ \dots terms of lower degree in $\beta$ \dots }).
    \end{align*}
This polynomial has only 35 distinct nonzero coefficients, given by
    \[
\left\{\ba &
1,\, 2,\, 4,\, 5,\, 7,\, 8,\, 9,\, 10,\, 12,\, 16,\, 24,\, 28,\,   30,\, 32,\, 34,\, 41,\, 43,\, 64,\, 65,\, 72,\, 80, \\& 
109,\, 110,\, 116,\, 121,\, 128,\, 142,\, 159,\, 173,\, 177,\, 180,\, 246,\, 261,\, 292,\, 344
\ea \right\}.
\]
The entries of $\fk M$ are \emph{not} all polynomials, although $\pf(\fk M)$ is a polynomial.

\end{example}

\section{Stable Grothendieck polynomials}

The limit of a sequence of polynomials or formal power series is defined to converge
if the sequence of coefficients of any fixed monomial is eventually constant.
Let $n \in \PP$ and $w \in S_n$. 
Given $m \in \NN$, define $1^m \times w  \in S_{m+n}$
to be the permutation that maps $i \mapsto i$ for $i \leq m$ and $i + m \mapsto w(i) + m$
for $i \in \PP$. 
The \emph{stable Grothendieck polynomial} of $w$ is then
\be\label{G-eq}
G_w := \lim_{n \to \infty} \fkG_{1^n \times w} \in \ZZ[\beta][[x_1,x_2,\dots]].
\ee
Remarkably, this limit always converges and the resulting power series
is a symmetric function in the $x_i$ variables with many notable properties \cite[\S2]{Buch2002}.
In this section, we study the
natural analogues of \eqref{G-eq} for
 orthogonal and symplectic Grothendieck polynomials.

\subsection{$K$-theoretic symmetric functions}

To begin, we review some properties of $G_w$ and related symmetric functions.
If $\lambda = (\lambda_1 \geq \lambda_2 \geq \dots \geq \lambda_k > 0)$ is an integer partition, then a \emph{set-valued tableau} of shape $\lambda$
is a map $T : (i,j) \mapsto T_{ij}$ from the Young diagram
\[\YY_\lambda := \{(i,j) \in \PP\times \PP : j \leq \lambda_i\}\]
to the set of finite, nonempty subsets of $\PP$.
For such a map $T$, define
\[x^T := \prod_{(i,j) \in \YY_\lambda} \prod_{k \in T_{ij}} x_k
\qquand |T| := \sum_{(i,j) \in \YY_\lambda} |T_{ij}|.\]
A set-valued tableau $T$ is \emph{semistandard} if one has $\max(T_{ij}) \leq \min(T_{i,j+1})$
and $\max(T_{ij}) < \min(T_{i+1,j})$ for all relevant $(i,j) \in \YY_\lambda$.
Let $\SetSSYT(\lambda)$ denote the set of semistandard set-valued tableaux of shape $\lambda$.

\begin{definition}
The \emph{stable Grothendieck polynomial} of a partition $\lambda$ is 
\[ G_\lambda := \sum_{T \in \SetSSYT(\lambda)} \beta^{|T|-|\lambda|} x^T\in \ZZ[\beta][[x_1,x_2,\dots]].\]
\end{definition}

This definition sometimes appears in the literature with the parameter $\beta$
set to $\pm 1$.
This specialization is immaterial to most results since 
if we write $G^{(\beta)}_\lambda = G_\lambda$ then 
$(-\beta)^{|\lambda|}G^{(\beta)}_\lambda = G^{(-1)}_\lambda(-\beta x_1,-\beta x_2,\dots)$.
%The results in \cite[\S3]{Buch2002} imply that 
%$G_\lambda$ is symmetric in the variables $x_1,x_2,\dots$, though of unbounded degree.
Setting $\beta=0$ transforms $G_\lambda$ to the usual Schur function $s_\lambda$.
%It follows that every power series in $\ZZ[\beta][[x_1,x_2,\dots]]$ that is symmetric in the $x_i$'s 
%can be written as a unique, possibly infinite $\ZZ[\beta]$-linear combination of $G_\lambda$'s.

%\begin{example}
%For the partition $\lambda = (1)$ with a single part we have 
%\[
%G_{(1)} = \sum_{m \in \PP} \sum_{1 \leq i_1 < i_2 <\dots < i_m} \beta^{n-1} x_{i_1}x_{i_2}\cdots x_{i_m}
%=
%s_{(1)} + \beta s_{(1,1)} + \beta^2 s_{(1,1,1)} + \dots.\]
%%where $1^n := (1,1,\dots,1)$ stands for the partition with $n$ parts of size 1.
%\end{example}

The symmetric functions $G_\lambda$ are related to
$G_w$ for $w \in S_n$ by the following theorems.
Given a partition $\lambda = (\lambda_1 \geq \lambda_2 \geq \dots \geq \lambda_k>0)$ 
with $k+\lambda_1 \leq n$,
define $w_\lambda \in S_n$ to be the unique permutation
with $w_\lambda(i) =i+ \lambda_{k+1-i}$ for $i \in [k]$ and $w_\lambda(i) < w_\lambda(i+1)$
for all $k<i\leq n$.
Write $\sP$ for the set of all partitions.

\begin{theorem}[{\cite[Theorem 3.1]{Buch2002}}]
\label{buch-thm1}
If $\lambda$ is any  partition then $G_{w_\lambda} = G_\lambda$.
\end{theorem}

\begin{theorem}[{\cite[Theorem 1]{BKSTY}}]
\label{buch-et-al-thm}
If $w \in S_n$ then 
$G_w \in \NN[\beta]\spanning\left\{ G_\lambda : \lambda \in \sP\right\}.$
\end{theorem}

%[Theorem 5.4]
Buch \cite{Buch2002} also derives a Littlewood-Richardson rule for
the stable Grothendieck polynomials $G_\lambda$,
which shows that the product $G_\lambda G_\mu$ is always a finite 
$\NN[\beta]$-linear combination of the functions $G_\nu$.

% \subsection{Shifted stable Grothendieck polynomials}

There are shifted analogues of $G_\lambda$ that will be 
related in a similar way to our orthogonal and symplectic analogues of \eqref{G-eq}.
Define the \emph{marked alphabet} to be totally ordered set 
%of primed and unprimed positive integers
 $\MM := \{1'<1<2'<2<\dots\}$,
and write
$|i'|  := |i| = i$ for $i \in \PP$.
If $\lambda =(\lambda_1>\lambda_2>\dots>\lambda_k>0)$ is a strict partition, then a \emph{shifted set-valued tableau} of shape $\lambda$
is a map $T : (i,j) \mapsto T_{ij}$ from the shifted diagram
\[\SS_\lambda := \{(i,i+j-1) \in \PP\times \PP : 1\leq j \leq \lambda_i\}\]
to the set of finite, nonempty subsets of  $\MM $.
Given such a map, define
\[x^T := \prod_{(i,j) \in \SS_\lambda} \prod_{k \in T_{ij}} x_{|k|}
\qquand
 |T| := \sum_{(i,j) \in \SS_\lambda} |T_{ij}|.\]
A shifted set-valued tableau $T$ is \emph{semistandard} if 
for all relevant $(i,j) \in \SS_\lambda$:
\begin{itemize}
\item[(a)] $\max(T_{ij}) \leq \min(T_{i,j+1})$ and $T_{ij} \cap T_{i,j+1} \subseteq \{1,2,3,\dots\}$.
\item[(b)] $\max(T_{ij}) \leq \min(T_{i+1,j})$ and $T_{ij} \cap T_{i+1,j} \subseteq \{1',2',3',\dots\}$.
\end{itemize}
In such tableaux, an unprimed number can appear at most once in a column, while a primed number 
can appear at most once in a row.
Let $\SetSSMT(\lambda)$ denote the set of semistandard shifted set-valued tableaux of shape $\lambda$.
\begin{definition}%[Ikeda and Naruse \cite{IkedaNaruse}]
The \emph{$K$-theoretic Schur $P$-function} and \emph{$K$-theoretic Schur $Q$-function}
of a strict partition $\lambda$ are the formal power series
\[
\GP_\lambda := \sum_{\substack{T \in \SetSSMT(\lambda) \\ T_{ii} \subseteq \PP\text{ if }(i,i) \in \SS_\lambda}} \beta^{|T|-|\lambda|} x^T
\quand
\GQ_\lambda := \sum_{T\in\SetSSMT(\lambda)} \beta^{|T|-|\lambda|} x^T.
\]
The summation defining $\GP_\lambda$ is over shifted set-valued tableaux with no primed numbers in any position on the main diagonal.
\end{definition}

These definitions are due to Ikeda and Naruse \cite{IkedaNaruse}, who
also show that
$\GP_\lambda$ and $\GQ_\lambda$ are symmetric in the $x_i$ variables
\cite[Theorem 9.1]{IkedaNaruse}.
Setting $\beta=0$ transforms $\GP_\lambda$ and $\GQ_\lambda$ 
to the \emph{Schur $P$-} and \emph{$Q$-functions} $P_\lambda$ and $Q_\lambda$.

\begin{example}
We have $\GP_{(1)} = G_{(1)} = s_{(1)} + \beta s_{(1,1)} + \beta^2 s_{(1,1,1)} + \dots $ while 
\[
\GQ_{(1)} = \sum_{m\in\PP} \sum_{1\leq i_1 < i_2 < \dots < i_m} \beta^{n-1} ( x_{i_1} \oplus x_{i_1})( x_{i_2} \oplus x_{i_2})\cdots (x_{i_m} \oplus x_{i_m})
\]
where $x\oplus y := x + y + \beta xy$ as in \eqref{oplus-def}.
\end{example}

Clifford, Thomas, and Yong prove a Littlewood-Richardson rule for the 
$\GP_\lambda$ functions in \cite{CTY},
which shows that each product $\GP_\lambda \GP_\mu$ is a finite $\NN[\beta]$-linear combination
of $\GP_\nu$ terms with positive coefficients; see the discussion in \cite[\S1]{HKPWZZ}.
A general Littlewood-Richardson rule for the $K$-theoretic Schur $Q$-functions $\GQ_\lambda$ is not yet known.
Each product $\GQ_\lambda \GQ_\mu$ is a linear combination of $\GQ_\nu$ terms
\cite[Proposition 3.5]{IkedaNaruse},
 but it is an open problem to
determine if these combinations are always finite \cite[Conjecture 3.2]{IkedaNaruse}.

\subsection{Orthogonal and symplectic variants}

Assume $n$ is even and let $z \in \ISp_n$
be a fixed-point-free involution in $S_n$.
Given $m \in \NN$, let
$(21)^{m} \times z = 21 \times 21 \times \cdots \times 21 \times z \in \ISp_{n+2m}$
denote  the involution 
 that maps $i \mapsto i - (-1)^i$ for $i \leq 2m$ and $i+2m \mapsto z(i) + 2m$ for $i \in \PP$.
We define the  \emph{symplectic stable Grothendieck polynomials} 
of $z$ to be the limit 
\be\label{gsp-eq}
\GSp_z := \lim_{m \to \infty} \Gfpf_{(21)^{m} \times z}.
\ee
These limits are always defined and have the following formula:

\begin{corollary}\label{gsp-cor}
If $z \in \ISp_n$ then  $\GSp_z = \sum_{w \in \HAfpf(z)} \beta^{\ell(w) - \ellfpf(z)} G_w.$
\end{corollary}

\begin{proof}
Proposition~\ref{ha-prop} implies that
$\HAfpf\( (21)^{m} \times z\) = \{ 1^{2m} \times w : w \in \HAfpf(z)\}$
for all $z \in \ISp_n$ 
and $m \in \PP$,
so this follows from Theorem~\ref{atom-thm}.
\end{proof}

Let $\SLambda$ denote the set of all strict partitions.
The symmetric functions $\GSp_z$ were studied in \cite{Mar},
which proves the following analogue of Theorem~\ref{buch-et-al-thm}:

\begin{theorem}[{\cite[Theorem 1.9]{Mar}}]
\label{mar-thm}
If $z \in \ISp_\infty$ then 
\[\GSp_z \in \NN[\beta]\spanning\left\{ \GP_\lambda : \lambda \in \SLambda\right\}.\]
\end{theorem}

There is also a symplectic analogue of Theorem~\ref{buch-thm1},
which shows that every $K$-theoretic Schur $P$-function occurs 
as $\GSp_z$ for some $n \in 2\PP$ and $z \in \ISp_n$; see \cite{MP}.
We mention one corollary of \cite[Theorem 1.9 and Corollary 3.27]{Mar}:

\begin{corollary}[See \cite{Mar}]
If $n \in 2\PP$ then $\GSp_{n\cdots 321} = \GP_{(n-2,n-4,n-6,\dots,2)}$.
\end{corollary}

For the rest of this section let $n \in \PP$ be arbitrary and suppose $z \in I_n$.
We wish to define the \emph{orthogonal stable Grothendieck polynomial} of $z$ by
\be
\GO_z := \lim_{m\to \infty} \iG_{1^m \times z}.
\ee
Unlike \eqref{gsp-eq},
it is not clear
that this limit exists for an arbitrary involution,
though we expect that this is always the case.

By  Theorem~\ref{thm:GO-pfaffian}, 
we at least know that $\GO_z$ is a well-defined power series when $z \in I_\infty$
is vexillary, 
since then $1^m\times z$ is also vexillary with 
\[ 
\DO(1^m \times z) = \{ (i+m, j+m) : (i,j)\in \DO(z)\}
\]
for all $m \in \NN$, so the corresponding sequence of Pfaffian
formulas for $\iG_{1^m\times z}$ obviously converges.
Since the matrix entries \eqref{fkm-eq} are symmetric when $p,q\to \infty$,
the power series $\GO_z$ is also symmetric when $z$ is vexillary.
Our last main result will show 
that in this case $\GO_z$ is actually a single $K$-theoretic Schur $Q$-function.

For this, we require the following theorem of 
Nakagawa and Naruse  \cite[Theorem 5.2.4]{NN2018}.
Write $\pf[a_{ij}]_{1\leq i < j\leq m}$ for the Pfaffian of the $m\times m$ 
skew symmetric matrix $A$ whose entries satisfy
$A_{ij}= -A_{ji}= a_{ij}$ for $i<j$.
Define 
\[ \Pi(u,v) := \frac{1}{1+\beta v} \prod_{j=1}^\infty \frac{1 + \beta x_j}{1-x_j u} \cdot (1 + (u+\beta) x_j) \in \ZZ[\beta][[u,v,x_1,x_2,\dots]]
\]
and 
\[ \Delta(u,v) := \frac{1-uv}{(1+\beta u)(1+\beta u + uv)} \in \ZZ[\beta][[u,v]].\]
Finally, for any $a,b \in \NN$ let
\[
\GQ_{(a,b)} := [u^{-a}v^{-b}]\Pi(u^{-1},u)\Pi(v^{-1},v)\Delta(u,v^{-1}).
\]

\begin{theorem}[{\cite[Theorem 5.7]{NN2018}}] \label{thm:nakagawa-naruse}
Let $\lambda = (\lambda_1>\lambda_2>\dots>\lambda_r\geq 0)$ be a strict partition with $r \in 2\PP$ parts, 
the last of which may be zero. Then 
\[ \GQ_\lambda = \pf\left[ \sum_{k=0}^\infty \sum_{l=0}^\infty \beta^{k+l}\binom{i+1-r}{k} \binom{j-r}{l} \GQ_{(\lambda_i+k,\lambda_j+l)} \right]_{1\leq i < j \leq r}.\]
\end{theorem}

In the next two results,
let $c^{(i)}(u) = \prod_{j= 1}^\infty \frac{1 + x_j u}{1 + \bar{x}_j u}$ for $i \in \PP$ where $\bar x := \frac{-x}{1+\beta x}$,
and define $R^{(i,j)}$ for $i,j \in \PP$ as in \eqref{Rij-eq}.

\begin{lemma} \label{lem:two-row-GQ} 
If $a,b \in \NN$ then
    \begin{equation*}
        \GQ_{(a,b)} =R^{(1,2)} \left (1-\beta T^{(1)}\right)^{-a+1}\left(1-\beta T^{(2)}\right)^{-b} c^{(1)}_a c^{(2)}_b.
    \end{equation*}
\end{lemma}

\begin{proof}
    Abbreviate by setting $T := T^{(1)}$ and $c(u) =\sum_{j\geq 0} c_j u^j:= c^{(i)}(u)$, and note that we 
    then have
    $\Pi(u,v) = \tfrac{1}{1+\beta v} \cdot c(u+\beta).$  We compute 
\be \label{eq:c-expansion}
\ba
        c(u+\beta) &= \sum_{j \geq 0}c_j \sum_{m=0}^j \tbinom{j}{m} \beta^{m-j} u^m  
        \\&= \sum_{m\geq 0} u^m \sum_{j \geq m} \tbinom{j}{m} c_j \beta^{m-j}\\
        &= \sum_{m\geq 0} u^m \sum_{j \geq m} \tbinom{j}{m} \beta^{j-m}T^{j-m} c_m 
        = \sum_{m \geq 0} (1-\beta T)^{-m-1} c_m u^m.
\ea
\ee
From this, it follows that if $i \in \ZZ$ then
\[
u^{i} c(u^{-1}+\beta) 
= 
(1-\beta T)^{-i}\sum_{m\geq -i}  (1-\beta T)^{-m-1} c_{m+i} u^{-m}.
\] 
Since $c_{m+i} = T^{i} c_m$ if $m\geq \max\{0,-i\}$,
we deduce that
\begin{equation} \label{eq:u-relation} 
[u^m]\Bigl( u^i c(u^{-1}+\beta)\Bigr) = [u^m] \Bigl( T^i(1-\beta T)^{-i} c(u^{-1}+\beta)\Bigr) \end{equation}
for all $i \in \ZZ$ and $m \leq 0$.

One can check that substituting $u \mapsto \frac{T^{(1)}}{1-\beta T^{(1)}}$ and $v \mapsto \frac{T^{(2)}}{1-\beta T^{(2)}}$ transforms
\[\Delta(u,v^{-1}) 
    \cdot \tfrac{1}{1+\beta u} \cdot  \tfrac{1}{1+\beta v}
  \quad\mapsto\quad R^{(1,2)}  \(1-\beta T^{(1)}\)^2  \(1 - \beta T^{(2)}\).\]
    Fix $a,b \in \NN$.
Since $\GQ_{(a,b)}$ is the coefficient of $u^{-a} v^{-b}$ in
    \[ 
     \Delta(u,v^{-1}) 
    \cdot \tfrac{1}{1+\beta u} \cdot  \tfrac{1}{1+\beta v}  \cdot  c^{(1)}(u^{-1}+\beta) 
      c^{(2)}(v^{-1}+\beta),
         \]
 it follows from \eqref{eq:u-relation}
that $\GQ_{(a,b)}$ is also the coefficient of $u^{-a} v^{-b}$ in
\[
R^{(1,2)}   \(1-\beta T^{(1)}\)^2   \(1 - \beta T^{(2)}\)  
c^{(1)}(u^{-1}+\beta)  
      c^{(2)}(v^{-1}+\beta).
      \]
The result is now clear after using \eqref{eq:c-expansion}
to rewrite this last expression as
$
\sum_{m,n\geq 0} R^{(1,2)}  \(1-\beta T^{(1)}\)^{-m+1}  \(1 - \beta T^{(2)}\)^{-n} 
c^{(1)}_m c^{(2)}_n u^{-m} v^{-n}.
      $
\end{proof}

We may now state our final theorem.

\begin{theorem}\label{final-thm}
If $z \in I_n$ is vexillary then 
$\GO_z = \GQ_{\shO(z)}$.
\end{theorem}

\begin{proof}
Fix a vexillary involution $z \in I_n$.
Let $\lambda=\shO(z)$ and
   define $r$ to be the smallest even integer with $r \geq \ell(\lambda)$. 
As noted at the beginning of this section,
 Theorem~\ref{thm:GO-pfaffian} implies
 that $\GO_z$ is
  the Pfaffian of the $r \times r$ skew-symmetric matrix whose $(i,j)$ entry for $i < j$ is 
\[
       R^{(i,j)} \left(1 - \beta T^{(i)}\right)^{r-i-\lambda_i} \left(1 - \beta T^{(j)}\right)^{r-j-\lambda_j} c^{(i)}_{\lambda_i} c^{(j)}_{\lambda_j}.
\]
Thus, it suffices to show that   $\GQ_{\lambda}$ is given by the same Pfaffian.
It follows from Theorem~\ref{thm:nakagawa-naruse}
 and Lemma~\ref{lem:two-row-GQ} that
 $\GQ_{\lambda}$ is the Pfaffian of the $r \times r$ skew-symmetric matrix whose $(i,j)$ entry for $i < j$ is
    \begin{equation} \label{eq:NN-entry}
        R^{(i,j)} \sum_{k,l\geq 0}\beta^{k+l} \tbinom{i+1-r}{k}\tbinom{j-r}{l} \left(1 - \beta T^{(i)}\right)^{-\lambda_i-k+1} \left(1 - \beta T^{(j)}\right)^{-\lambda_j-l} c^{(i)}_{\lambda_i+k} c^{(j)}_{\lambda_j+l}.
    \end{equation}
But we have
    \begin{align*}
        & \sum_{k\geq 0} \beta^k \tbinom{i+1-r}{k} \left(1 - \beta T^{(i)}\right)^{-\lambda_i-k+1} c^{(i)}_{\lambda_i+k}\\
        &=\( \sum_{k\geq 0} \left(\beta T^{(i)}\right)^k \tbinom{i+1-r}{k} \left(1 - \beta T^{(i)}\right)^{i+1-r-k}\) \left(1 - \beta T^{(i)}\right)^{r-i-\lambda_i} c^{(i)}_{\lambda_i}\\
        &= \left(1-\beta T^{(i)}+\beta T^{(i)}\right)^{i+1-r} \left(1 - \beta T^{(i)}\right)^{r-i-\lambda_i} c^{(i)}_{\lambda_i}\\
        &= \left(1 - \beta T^{(i)}\right)^{r-i-\lambda_i} c^{(i)}_{\lambda_i}
    \end{align*}
 and similarly
 \[
 \sum_{l\geq 0} \beta^l \tbinom{j-r}{l} \(1 - \beta T^{(j)}\)^{-\lambda_j-l} c^{(j)}_{\lambda_j+l}
 =\(1 - \beta T^{(j)}\)^{r-j-\lambda_j}  c^{(j)}_{\lambda_j}.
 \]
Thus \eqref{eq:NN-entry} is equal to
$
        R^{(i,j)}\left(1 - \beta T^{(i)}\right)^{r-i-\lambda_i} \left(1 - \beta T^{(j)}\right)^{r-j-\lambda_j} c^{(i)}_{\lambda_i} c^{(j)}_{\lambda_j}
$
which suffices to prove the theorem.
\end{proof}

\begin{corollary}
If $n \in \PP$ then $\GO_{n\cdots 321} = \GQ_{(n-1,n-3,n-5,\dots)}$.
\end{corollary}

\begin{proof}
It suffices to observe that $\shO(n\cdots 321) = (n-1,n-3,n-5,\dots)$.
\end{proof}

Following \cite{HMP4}, we say that an involution $z \in I_n$ is \emph{I-Grassmannian}
if there are integers $r \in \NN$ and $1 \leq \phi_1  < \phi_2 < \dots < \phi_r \leq n$
such that 
\be\label{igrass-eq} z = (\phi_1,n+1)(\phi_2,n+2)\cdots(\phi_r,n+r).\ee
The case $n=r=0$ corresponds to $z=1$.
Computing $\shO(z)$  gives the following:

\begin{corollary}
If $z \in I_n$ is I-Grassmannian of the form \eqref{igrass-eq},
then 
\[\GO_{z} = \GQ_{(n+1-\phi_1,n+1-\phi_2,\dots,n+1-\phi_r)}.\]
\end{corollary}

Thus, every $K$-theoretic Schur $Q$-function occurs 
as $\GO_z$ for some $z$,
since for any strict partition $\lambda$ there is an I-Grassmannian involution
of shape $\lambda$.

\section{Open problems}\label{last-sect}

We conclude with a list of related open problems.

Each $\Gfpf_z$ is a finite 
linear combination of the polynomials $\fkG_w$; the summands are described by Proposition~\ref{ha-prop}.
It remains find analogous results for $\iG_z$:

\begin{problem}
Describe the set of summands expanding $\iG_z$ as a $\ZZ[\beta]$-linear combination of the polynomials $\fkG_w$.
Are the coefficients in this expansion all nonnegative?
\end{problem}

Lenart \cite{Lenart} proves a ``transition formula'' 
which expands $(1+\beta x_j) \fkG_w$
as a finite, $\NN[\beta]$-linear combination of the polynomials $\fkG_v$.
The sequel to this paper \cite{MP} describes an analogous formula 
involving the symplectic Grothendieck polynomials $\Gfpf_z$.

\begin{problem}
Is there a transition formula in the sense of \cite{Lenart,MP} for $\iG_z$?
\end{problem}

It remains to show that  $\GO_z$ is well-defined with $z \in I_n$ is not vexillary.

\begin{problem}
Show that $\GO_z := \lim_{m \to \infty} \iG_{1^m \times z}$ 
converges for all $z \in I_n$.
\end{problem}

Recall that $\sP$ and $\SLambda$ denote the sets of arbitrary and strict partitions.

\begin{problem}
Does it always hold that $\GO_z  \in \bigoplus_{\lambda\in\SLambda} \NN[\beta] \GQ_\lambda$?
\end{problem}

It is known that if $\lambda,\mu\in\SLambda$ then
$\GQ_\lambda \GQ_\mu \in \sum_{\nu\in\SLambda} \ZZ[\beta] \GQ_\nu$,
where the sum could involve
infinitely many terms $\GQ_\nu$.
The following problem, asserting that the sum is always finite, is
\cite[Conjecture 3.2]{IkedaNaruse}.

\begin{problem}
Show that if $\lambda,\mu\in\SLambda$ then
$\GQ_\lambda \GQ_\mu \in \bigoplus_{\nu\in\SLambda} \NN[\beta] \GQ_\nu$.
\end{problem}

If $\lambda\in\SLambda$ then $\GQ_\lambda \in \sum_{\mu\in\sP} \ZZ[\beta] G_\mu$
since this is true with $\GQ_\lambda$ replaced by any power series in $\ZZ[\beta][[x_1,x_2,\dots]]$
that is symmetric in the $x_i$ variables.
This expansion could be an infinite sum, but we expect that it is also finite:

\begin{problem}
Show that if $\lambda$ is a strict partition then $\GQ_\lambda \in \bigoplus_{\mu\in\sP} \NN[\beta] G_\mu$.
\end{problem}

\end{document}